\documentclass[11pt]{amsart} \textwidth=14.5cm \oddsidemargin=1cm
\usepackage{amsmath,amsthm,amssymb,latexsym,epic,bbm,comment,yfonts,mathrsfs}
\usepackage{graphicx,enumerate,stmaryrd,rotating,xcolor}
\usepackage[all]{xy}
\xyoption{2cell}

\allowdisplaybreaks

\usepackage[active]{srcltx}
\usepackage{youngtab}

\usepackage[latin1]{inputenc}
\usepackage{amsxtra}
\usepackage{amsmath}
\usepackage{amscd}
\usepackage{amssymb}
\usepackage{amsfonts}
\usepackage[all]{xy}
\usepackage{mathrsfs}
\usepackage{amsthm}
\usepackage{color}
\usepackage{upgreek}
\usepackage{verbatim}  
\usepackage{enumerate}
\usepackage{latexsym}
\usepackage{graphicx}
\usepackage{tikz}
\usepackage{todonotes}
\usepackage{setspace}
\usepackage{hyperref}
\usepackage{stmaryrd}
\usepackage{nicefrac}
\usepackage{euscript}
\usetikzlibrary{decorations.pathmorphing}

 \definecolor{darkgreen}{HTML}{336633}
 \definecolor{darkred}{HTML}{993333}
\makeatletter

\newcommand{\arxiv}[1]{\href{http://arxiv.org/abs/#1}{\tt
    arXiv:\nolinkurl{#1}}}

\theoremstyle{plain}
\newtheorem{thm}{Theorem}
\newtheorem*{thm*}{Theorem}

\newtheorem{lem}[thm]{Lemma}
\newtheorem{prop}[thm]{Proposition}

\newtheorem{cor}[thm]{Corollary}
\newtheorem{df-prop}[thm]{Definition-Proposition}

\theoremstyle{definition}

\theoremstyle{remark}
\newtheorem{rem}[thm]{Remark}

\def\bbI{\mathbb{I}}

\def\bbN{\mathbb{N}}

\def\bbQ{\mathbb{Q}}

\def\bbT{\mathbb{T}}

\def\bbV{\mathbb{V}}
\def\bbW{\mathbb{W}}

\def\bbZ{\mathbb{Z}}

\def\frakb{\mathfrak{b}}

\def\frakg{\mathfrak{g}}
\def\frakh{\mathfrak{h}}

\def\frakl{\mathfrak{l}}

\def\frakn{\mathfrak{n}}

\def\frakp{\mathfrak{p}}

\def\fraku{\mathfrak{u}}

\def\onto{\twoheadrightarrow}

\def\mod{\operatorname{-mod}\nolimits}

\newcommand{\hf}{{\Small \frac12}}

\def\Hom{\operatorname{Hom}\nolimits}

\def\Res{\operatorname{Res}\nolimits}
\def\Ind{\operatorname{Ind}\nolimits}

\def\wbT{\widehat{\mathbb{T}}}

\def\ep{\epsilon}
\def\gl{\mathfrak{gl}}

\def\la{\lambda}

\def\pn{\mf{p} (n)}

\def\ov{\overline}

\newcommand{\mc}{\mathcal}
\newcommand{\mf}{\mathfrak}
\newcommand{\C}{\mathbb C}

\newcommand{\oa}{{\bar 0}}
\newcommand{\ob}{{\bar 1}}

\newcommand{\fg}{\mathfrak{g}}

\newcommand{\fh}{\mathfrak{h}}

\newcommand{\n}{\mathfrak{n}}

\newcommand{\h}{\mathfrak{h}}

\newcommand{\ch}{\mathrm{ch}}

\newcommand{\Coind}{{\rm Coind}}

\newcommand{\g}{\mathfrak{g}}
\newcommand{\fl}{\mathfrak{l}}

\newcommand{\fu}{\mathfrak{u}}
\newcommand{\Real}{\mathrm{Re}}

\newcommand{\Z}{{\mathbb Z}}

     \def\Ann{{\text{Ann}}}

       \def\Id{{\text{Id}}}

          \def\OI{{\ov{\mc O}}}
         \def\Lnua{{\Lambda(\zeta)}}

           \def\plz{{P_{\mf l_\zeta}}}
               \def\Opres{\mc O^{\zeta\text{-pres}}}
               \def\vpre{{\zeta}\text{-pres}}
\begin{document}

\numberwithin{equation}{section}

\title[Whittaker categories and quantum symmetric pairs]{Whittaker categories of quasi-reductive Lie superalgebras and quantum symmetric pairs}

\author[Chen]{Chih-Whi Chen}
\address{Department of Mathematics, National Central University, Chung-Li, Taiwan 32054} \email{cwchen@math.ncu.edu.tw}
\author[Cheng]{Shun-Jen Cheng}
\address{Institute of Mathematics, Academia Sinica, and National Center of Theoretical Sciences, Taipei, Taiwan 10617} \email{chengsj@math.sinica.edu.tw}
\date{}

\begin{abstract} 
	
	We show that, for an arbitrary quasi-reductive Lie superalgebra with a triangular decomposition and a character $\zeta$ of the nilpotent radical, the associated Backelin functor $\Gamma_\zeta$ sends Verma modules to standard Whittaker modules provided the latter exist.  As a consequence, this gives a complete solution to the problem of determining the composition factors of the standard Whittaker modules in terms of composition factors of Verma modules in the category $\mc O$.  
 In the case of the ortho-symplectic Lie superalgebras, we show that the Backelin functor $\Gamma_\zeta$ and its target category, respectively, categorify a $q$-symmetrizing map and the corresponding $q$-symmetrized Fock space associated with a quasi-split quantum symmetric pair of type $AIII$.
	
\end{abstract}

\maketitle

\tableofcontents


\vspace{5mm}

\section{Introduction and Description of Results}\label{sec1}
\subsection{Background and motivation}
Let $\g$ be a finite-dimensional quasi-reductive Lie superalgebra with a triangular decomposition
	\begin{align}
		&\g=\mf n^- \oplus \h \oplus \mf n^+ \label{eq::tri}
	\end{align}  in the sense of \cite{Ma} (see also \cite{CCC}). Here $\h$ is a purely even Cartan subalgebra with nil- and opposite radical $\mf n^\pm$, respectively. 
Let $Z(\g_\oa)$ be the center of the universal enveloping algebra  $U(\g_\oa)$ of $\g_\oa$. Motivated by earlier results of Mili{\v{c}}i{\'c}-Soergel \cite{MS} and Backelin \cite{B}, Mazorchuk and the authors \cite{Ch21,Ch212,CCM2} studied the category $\mc N$ consisting of finitely generated $\g$-modules that are locally finite over $Z(\g_\oa)$ and $U(\mf n^+)$, which we refer to as {\em Whittaker modules} over $\g$.   The category $\mc N$ decomposes into a direct sum of the full subcategories $\mc N(\zeta)$, with $\zeta\in\ch\mf n^+_\oa:{=}\left(\mf n_\oa^+/[\mf n_\oa^+,\mf n_\oa^+]\right)^*$, consisting of modules $M\in \mc N$ on which $x-\zeta(x)$ acts locally nilpotently, for any $x\in \mf n^+_\oa$. Furthermore, the category $\mc N(\zeta)$ is a union of certain full subcategories $\mc N(\zeta)^n$, for $n\geq 1$, in the sense of \cite[Section 4.1.2]{Ch21}, which extends the definition in \cite[Section 5]{MS} for Lie algebras. In particular, $\mc N(\zeta)^1$ contains all simple objects in $\mc N(\zeta)$ and $\mc N(0)^1$ coincides with the Bernstein-Gelfand-Gelfand (BGG) category $\mc O$. Therefore, the category $\mc N(\zeta)^1$ provides a suitable framework for the study of Whittaker modules. In the present paper, we  introduce an integral block $\mc  W(\zeta)$ of $\mc N(\zeta)^1$ (see Section \ref{sect::23}). In particular, $\mc W(0)$ coincides with the integral BGG category $\mc O_\Z$.

For Lie superalgebras, the {\em standard Whittaker modules} $M(\la,\zeta)\in \mc N(\zeta)$ were introduced in  \cite{BCW,Ch21}, $\la\in \h^\ast$. These modules were initially studied by McDowell \cite{Mc,Mc2} and  Mili{\v{c}}i{\'c}-Soergel \cite{MS} for Lie algebras. They are generalizations of Verma modules and play an analogous role in the representation theory of $\mc W(\zeta)$.

A quasi-reductive Lie superalgebra $\g$ is said to be of {\em type I} if $\g$ has a $\Z_2$-compatible $\Z$-gradation  of the form $\g=\g_{-1}\oplus \g_0\oplus \g_1$. In this case,  the standard Whittaker modules and the Whittaker category $\mc W(\zeta)$ have been investigated by the first author \cite{Ch21}, and then by Mazorchuk and the authors \cite{CCM2}. In particular, for the general linear Lie superalgebra $\gl(m|n)$ the composition multiplicity of standard Whittaker modules are controlled by the Brundan-Kazhdan-Lusztig polynomials from \cite{Br1, CLW2}. Furthermore, it was shown in \cite{CCM2} that $\mc W(\zeta)$ categorifies a $q$-symmetrized Fock space over a quantum group of type A in the case $\g=\gl(m|n)$.

A main motivation for the present paper is to study the multiplicities of composition factors of standard Whittaker modules and a corresponding categorification picture using the Whittaker categories $\mc W(\zeta)$ for Lie superalgebras beyond type I.

\subsection{Setup}   \label{sect::assumption}
Throughout this paper, we let $\g$ be  an arbitrary finite-dimensional quasi-reductive Lie superalgebra.  This means  that $\g_\oa$
	is a reductive Lie algebra and $\g_\ob$
	is semisimple as a $\g_\oa$-module. We fix a triangular decomposition of $\g$ as given in \eqref{eq::tri}. The corresponding {\em Borel subalgebra} is defined to be $\mf b:=\mf h\oplus \mf n^+$.  We are mainly interested in the following examples of quasi-reductive Lie superalgebras in Kac's list \cite{Ka1}:
\begin{align} \label{eq::Kaclist}
&\gl(m|n),~{\mf{sl}}(m|n),~{\mf{osp}}(m|2n), ~D(2,1;\alpha),~ G(3), ~F(4), ~\mf{p}(n).
\end{align}  We refer to, e.g., \cite{ChWa12, Mu12} for more details about these Lie superalgebras.

For each root $\alpha$ of $\g$, we denote by $\g_\alpha:=\{x\in \g|~ hx=\alpha(h)x, \text{ for all }h\in \h\}$  the corresponding root space. For $\zeta\in\ch\mf n_\oa^+$ we denote by $\Pi_\zeta$  the set of all roots $\alpha$ in $\mf n^+_\oa$ with $\zeta(\g_\alpha)\neq 0.$ This leads to a Levi subalgebra $\mf l_\zeta$ of $\g_\oa$, that is, $\mf l_\zeta$ is generated by $\mf h$ and $\g_{\pm\alpha}$ for $\alpha \in \Pi_\zeta$.

In this paper, unless otherwise specified, we shall assume that  $\mf l_\zeta$ is a Levi subalgebra in a parabolic decomposition in the sense of \cite[Section 2.4]{Ma} (in particular, $\zeta$ can be extended trivially to a character of $\mf n^+$):
	\begin{align}
	&\mf g=\mf u^-\oplus \mf l_\zeta \oplus \mf u^+,
	\end{align} which is compatible with the triangular decomposition given in \eqref{eq::tri}, namely,  $\mf u^\pm \subseteq \mf n^\pm$. That is,  there exists $H\in \fh$  such that
\begin{align} \label{eq::paracomp}
\fl_\zeta:=\bigoplus_{\Real \alpha(H)=0} \fg_\alpha,\quad\fu^+:=\bigoplus_{\Real \alpha(H)>0} \fg_\alpha,\quad\fu^-:=\bigoplus_{\Real \alpha(H)<0} \fg_\alpha,
\end{align}
such that $\Real\alpha(H)\geq 0,~\text{for any root } \alpha$ in $\mf n^+,$
where $\Real(z)$ denotes the real part of $z\in \mathbb C$;  see also \cite[Section 3.3]{Ch21} for examples. We let $\frakp:=\frakl_\zeta+\fraku^+$ be the corresponding parabolic subalgebra.
  We remark that, if one starts with a character $\zeta\in \ch {\mf n^+_\oa}$ which leads to a parabolic decomposition in \eqref{eq::paracomp}, then there exists a compatible  triangular decomposition of $\g$; see \cite[Lemma 1.3]{CCC}.
  We note that $H$ lies in the center of $\mf l_\zeta$.    In this article we exclude the family of queer Lie superalgebras.

 Motivated by the equivalences of categories in the earlier works of Bernstein-Gelfand \cite{BG} and  Mili{\v{c}}i{\'c}-Soergel \cite{MS}, which connect modules in $\mc O$ and Whittaker modules in $\mc N$, an equivalence is established in \cite{Ch21, Ch212} between $\mc W(\zeta)$ and a certain full subcategory $\mc \Opres$ of $\mc O_\Z$ (see Section \ref{Sect::Opres}). The objects in $\mc \Opres$ admit a two step resolution by projective modules whose simple quotients are free $U(\g_{-\alpha})$-modules, for any $\alpha\in \Pi_\zeta$. 

  \subsection{Goal} An important task in the study of Whittaker modules is to find the composition factors of a given standard Whittaker module. In the case of reductive Lie algebras,  Mili{\v{c}}i{\'c} and Soergel developed in \cite{MS}  an equivalence between $\mc N(\zeta)$ and certain category of Harish-Chandra bimodules to provide an answer for the composition multiplicity of $M(\la,\zeta)$ for regular and integral weights $\la$.  Subsequently, Backelin in \cite{B} introduced an exact functor $\Gamma_\zeta$ from  $\mc O$ to $\mc N(\zeta)$, sending Verma modules to standard Whittaker modules, to give a complete solution in terms of Kazhdan-Lusztig polynomials.

Standard Whittaker modules for Lie superalgebras of type I  was originally considered in \cite{BCW}, which are of finite length.  However, as a complete classification of simple Whittaker modules was not available at that time, limited effort had been made to resolve the problem of finding composition multiplicity for standard Whittaker modules over Lie superalgebras  until recently. It is proved in \cite{Ch21} that in type I case this multiplicity problem reduces to that of Verma modules in $\mc O$; see also \cite{BrG} for a special non-singular case. In particular, in the case of $\g=\gl(m|n)$ these multiplicities can be computed by means of Brundan-Kazhdan-Lusztig polynomials \cite{Br, CLW2}; see also \cite{CCM2}.

When the Lie superalgebra is not of type I standard Whittaker modules similarly as in McDowell \cite{Mc,Mc2} can also be defined under the assumption that $\zeta$ satisfies the assumption in Section \ref{sect::assumption}. This is done in \cite{Ch21} and also in the present paper.

   By \cite[Theorem 6]{Ch21},  for any $\la \in \mf h^\ast$, the top of the standard Whittaker $M(\la, \zeta)$ is simple, which we  denote by $L(\la,\zeta)$. These simple Whittaker modules are classified in terms of orbits in $\h^\ast$ under the dot-action of the Weyl group $W_\zeta$ of $\mf l_\zeta$, and they constitute all simple objects in $\mc N(\zeta)$. One of the main objectives of this paper is to provide a complete answer to the composition multiplicity problem for standard Whittaker modules over quasi-reductive Lie superalgebras beyond type I.


 \subsection{Main results}
  For each $M\in \mc O$, we denote the completion of $M$ with respect to its weight spaces by  $\ov M:=\prod_{\la\in \mf h^\ast} M_\la$. Then, the module $\Gamma_\zeta(M)$ is defined to be the submodule of $\ov{M}$ consisting of vectors annihilated by a power of $x-\zeta(x)$, for all $x\in\mf n^+_\oa$. Then $\Gamma_\zeta$ defines an exact functor  from $\mc O$ to $\mc N(\zeta)$, which we refer to as the {\em Backelin functor} associated to $\zeta$ for $\g$. Let $M(\la)$ be the Verma module of highest weight $\la$ with respect to the triangular decomposition \eqref{eq::tri} and $L(\la)$ its unique simple quotient.
   Our   first main result is the following.
\begin{thm}[Proposition \ref{prop::3}, Theorem \ref{thm::4}-(ii)] \label{thm::mainthm}
	For any weight $\la\in \h^\ast$, the Backelin functor  $\Gamma_\zeta(-)$ sends $M(\la)$ to $M(\la,\zeta)$. Furthermore, $\Gamma_\zeta(-)$ sends $L(\la)$ to $L(\la,\zeta)$, if $\la$ is $W_\zeta$-anti-dominant, and to zero otherwise.
	
	Consequently, we have the following multiplicity formula:
	\begin{align}
		&[M(\la,\zeta): L(\mu,\zeta)] = [M(\la): L(\mu)],
	\end{align} for $W_\zeta$-anti-dominant weights $\la, \mu$.
\end{thm}

Any Lie superalgebra from the series $\gl(m|n),~{\mf{sl}}(m|n),~{\mf{osp}}(2|2n)$ and $\mf p(n)$ in \eqref{eq::Kaclist} admits a $\Z_2$-compatible gradation $\g= \g_{-1}\oplus \g_0 \oplus\g_1$,  i.e., they are Lie superalgebras of  type I. In this case, Theorem \ref{thm::mainthm} recovers (and provides a uniform proof of) \cite[Thoerem C]{Ch21}.

Recently, an $\imath$-{\em canonical basis} theory based on a quantum symmetric pair $(U,U^\imath)$ has been systematically developed by Bao and Wang starting with the pioneering work \cite{BaoWang}, which extends Lusztig's canonical basis theory for quantum groups \cite{Lu90,Lu92}. The main motivation of \cite{BaoWang} was to resolve the irreducible character problem in  the integral BGG category $\mc O_\Z$ for the ortho-symplectic Lie superalgebra. For $\mf{osp}(2m+1|2n)$ the $\imath$-canonical basis theory of the associated $\imath$-quantum group $U^\imath$ (see Section \ref{sec:QSP:def}) on a completed Fock space $\wbT^{m|n}$ (see Section \ref{sec:tmn}) provides a complete solution to this problem in subcategory $\mc O^\texttt{int}_\Z$ of integer weight modules (see Section \ref{sec:ortho}). In this case, the corresponding quantum symmetric pair $(U,U^\imath)$ is quasi-split of type $AIII$. In particular, the linear isomorphism $\psi$ from $\wbT^{m|n}$ to a corresponding completed Grothendieck group $\widehat{K}(\mc O^\texttt{int}_\bbZ)_\bbQ$, sending standard monomial basis elements to Verma modules, match the $\imath$-canonical and dual $\imath$-canonical bases with tilting and simple objects, respectively. A similar solution to the irreducible character problem in the subcategory $\mc O^\texttt{hf}_\Z$ of half-integer weight modules (see again Section \ref{sec:ortho}) is given in terms of so-called $\jmath$-canonical basis \cite[Chapter 12]{BaoWang}. Note that this together solves the irreducible character problem in $\mc O_\Z$, as other integral weights are typical.

In the present paper we study a certain $q$-symmetrized Fock space $\wbT^{m|n}_{\zeta}$, depending on the character $\zeta$, which can be regarded either as a $U^\imath$-submodule or as a $U^\imath$-quotient of $\wbT^{m|n}$. There is a canonical $q$-symmetrizer map $\phi_\zeta:\wbT^{m|n}\rightarrow\wbT^{m|n}_\zeta$. Building on the works of \cite{BaoWang}, we construct in a natural way $\imath$-canonical and dual $\imath$-canonical bases on this $q$-symmetrized Fock space. We show that the $U^\imath$-module $\wbT^{m|n}_\zeta$ is categorified by an ``integer weight'' subcategory $\mc W(\zeta)^\texttt{int}$ (see Section \ref{Sect::thm:match:can}) of our category $\mc W(\zeta)$, when $\mf g$ is the ortho-symplectic Lie superalgebra $\mf{osp}(2m+1|2n)$. The linear action of $U^\imath$ on the Fock space at $q=1$ translates to action of certain translation functors on the category. We denote by $\psi_\zeta$ the induced isomorphism from $\wbT^{m|n}_{\zeta,q=1}$ to the completed Grothendieck group $\widehat{K}(\mc W(\zeta)^{\texttt{int}})_\bbQ$. We can now state our second main result, which is equivalent to Theorem \ref{thm:match:can} below.

\begin{thm} \label{thm::mainthmB} Let $\g$ be $\mf{osp}(2m+1|2n)$. Let $\gamma_\zeta:\widehat{K}(\mc O^{\texttt{int}}_\bbZ)_\bbQ\rightarrow \widehat{K}(\mc W(\zeta)^{\texttt{int}})_\bbQ$ be the map induced by the Backelin functor $\Gamma_\zeta$.  We have the following commutative diagram of $U^\imath_{q=1}$-homomorphisms:	
\begin{eqnarray}\label{comm:diag}
\CD
\xymatrixcolsep{3pc} \xymatrix{
  	\widehat{\mathbb T}^{m|n}_{q=1}  \ar[r]^-{\psi}  \ar@<-2pt>[d]_{\phi_\zeta} &   \widehat{K}(\mc O^{\texttt{int}}_\Z)_\bbQ  \ar@<-2pt>[d]^{\gamma_\zeta} \\
  	\widehat{\mathbb T}^{m|n}_{\zeta,q=1} \ar[r]^-{\psi_\zeta}  &  \widehat{K}(\mc W(\zeta)^{\texttt{int}})_\bbQ}
\endCD
\end{eqnarray}
Furthermore, the map $\psi_\zeta$ matches $\imath$-canonical and dual $\imath$-canonical basis with the tilting and simple objects in $\mc W(\zeta)^{\texttt{int}}$, respectively.
	
\end{thm}

There is also an analogue of Theorem \ref{thm::mainthmB} in the case of half-integer weight subcategory $\mc W(\zeta)^\texttt{hf}$ of $\mc W(\zeta)$ formulated in Theorem \ref{thm:match:can1}. In this case, it is formulated in term of so-called $\jmath$-canonical basis. Also, we have counterparts for the Lie superalgebra $\mf{osp}(2m|2n)$. Therefore, Theorem \ref{thm::mainthmB}, its half-integer and their type $D$ counterparts together give an ortho-symplectic analogue of the results in \cite[Section 3.7]{CCM2} for type $A$ Lie superalgebras.

\subsection{Additional results}

\subsubsection{ Homomorphisms from $M(\la,\zeta)$ to full dual of Verma modules}
Suppose that $\g$ is a  Lie superalgebra in Kac's list \eqref{eq::Kaclist} with $\g\neq \pn$. For any $\la\in \h^\ast$, we define $N(\la)$ to be the Verma module of lowest weight $\la$ with respect to $\mf b$, namely, $N(\la):=\Ind_{\h+\mf n^-}^\g\C_\la$. We note that $N(-\la)^\ast\cong \Coind^\g_{\mf h+\mf n^-}\C_\la$; see, e.g., \cite[Chapter 4]{Sch}. Also, we let $M(\la)^{\ast,\tau}$ be the full dual of the Verma module $M(\la)$ with the action of $\g$ twisted by the Chevalley
antiautomorphism $\tau$ on $\g$.

In Theorem \ref{thm::4}(i)  we prove that, for $\la,\mu\in \h^\ast$ and  $M = N(-\mu)^\ast$ or $M(\mu)^{\ast,\tau}$,  we have 	\begin{align}
		&\Hom_\g(M(\la,\zeta), M)= \left\{\begin{array}{ll}
			\C, &  \text{~for } \mu\in W_\zeta\cdot \la;\\
			0, & \text{ otherwise}.
		\end{array} \right.
	\end{align}
Here $W_\zeta\cdot \la$ denotes the orbit of $\la$ under the dot-action $\cdot$ of the Weyl group $W_\zeta$ of $\mf l_\zeta$ on $\mf h^\ast$.
In particular, this extends  \cite[Theorem 5.2]{BR}, where the case of semisimple Lie algebras with regular weight $\la$ was considered.

\subsubsection{A realization of the Serre quotient functor} 

The categories $\mc W(\zeta)$ and $\Opres$ fit naturally into the framework of the Serre quotient category $\OI$ in the sense of Gabriel  \cite{Gabriel}, which is a localization of $\mc O_\Z$ with the canonical quotient functor $\pi: \mc O_\Z\rightarrow \OI$ with respect to the set of homomorphisms with kernels and cokernels lying in a given Serre subcategory. In the case of Lie superalgebras of type I, it is proved in \cite{CCM2} that the Backelin functor $\Gamma_\zeta(-)$ is a realization of the quotient functor $\pi$. In this paper, we extend this result to any Lie superalgebras $\g$ in Kac's list \eqref{eq::Kaclist}. 

\subsubsection{Annihilator ideals of simple Whittaker modules}
 It is proved that every primitive ideal of a quasi-reductive Lie superalgebra is the annihilator of some simple highest weight modules; see, e.g.,  \cite{Duflo, Le89, Mu92, CC}.  In the case of $W_\zeta=W$,  Kostant \cite[Theorem 3.9]{Ko78} give a description of the annihilator ideals for simple Whittaker modules over semisimple Lie algebras. This has been extended to Lie superalgebras of type I in \cite[Theorem B]{Ch212}. In this paper, we will  prove an analogue for any Lie superalgebra $\g$ in Kac's list \eqref{eq::Kaclist}.

	\subsection{Organization} The paper is organized as follows. In Section \ref{Sect::2}, we provide the necessary preliminaries. In particular, we introduce various (equivalent) abelian categories, including the Whittaker categories $\mc W(\zeta)$. In Section \ref{Sect::3}, we give a proof of Theorem \ref{thm::mainthm}.
	Section \ref{Sect::4} is devoted to a description of the stratified structure on $\mc W(\zeta)$.  In particular, we describe explicitly the tilting modules in $\mc W(\zeta)$ and prove a BGG type reciprocity for $\mc W(\zeta)$.  In Section \ref{Sect::QSPcan} we study a $q$-symmetrized Fock space over the $\imath$-quantum group mentioned earlier and construct its $\imath$-canonical and the dual $\imath$-canonical bases on this space.  We briefly recall in Section \ref{sec:ortho} the solution of the irreducible character problem for the ortho-symplectic Lie superalgebras in category $\mc O$, which we then use in Section \ref{Sect::7} to complete the proof of Theorem \ref{thm::mainthmB}.

\vskip 0.5cm

{\bf Acknowledgment}. The authors are partially supported by National Science and Technology Council grants of the R.O.C., and they further acknowledge support from the National Center for Theoretical Sciences.

\section{Preliminaries} \label{Sect::2}
In this section, we review the necessary background and definitions that will be used in the remainder of the paper.

\subsection{Root systems and Weyl group}
 Let $\Phi$ be the set of all roots. Denote by $\Pi$, $\Phi^\pm$, $\Phi_{\oa}$ and $\Phi_\ob$ the sets of simple, positive/negative, even and odd roots of $\g$, respectively. Define $\Phi_\epsilon^\pm:=\Phi_\epsilon\cap \Phi^\pm$, $\epsilon=\bar{0},\bar{1}$. Set $\Pi_\oa$ to be the simple system for $\Phi_\oa^+$. For each $\alpha \in \Phi$, we denote by $\g_\alpha$ the root space of $\g$ corresponding to $\alpha$. The corresponding {\em Borel subalgebra} is $\mf b=\mf h\oplus \mf n^+$, with $\mf n^+=\bigoplus_{\alpha\in\Phi^+}\g_\alpha$.

Recall that we fixed a character $\zeta$ that gives rise to a Levi subalgebra $\mf l_\zeta$ of $\g$ as given in Section \ref{sect::assumption}.
 Also, recall that the set of all simple roots of $\mf l_\zeta$ is denoted as follows \begin{align*}
 	\Pi_\zeta:=\{\alpha\in \Pi_\oa|~\zeta(\mf g_\alpha)\neq 0\}.
 \end{align*}
We note that $\zeta([\mf n^+_\ob, \mf n^+_\ob])=0$, the latter implies that $\zeta$ extends trivially to a one-dimensional $\mf n^+$-module. We also denote by $\zeta$ the induced algebra homomorphism from $U(\mf n^+)$ to $\C$.

The Weyl group $W$ is defined to be the Weyl group of $\mf g_\oa$, and it acts naturally on $\h^\ast$. For each $\alpha \in \Pi_\oa$, let $s_\alpha\in W$ be the simple reflection corresponding to $\alpha$. We consider the {\em dot-action} of $W$ on $\h^\ast$, that is, $w\cdot \la:= w(\la+\rho_\oa)-\rho_\oa$, for any $w\in W$ and $\la \in \h^\ast$. Here $\rho_\oa$ denotes the half-sum of all positive even roots.  Also, we let $W_\zeta$ denote the Weyl group of $\mf l_\zeta$. We denote by $w_0$ and $w_0^\zeta$ the longest elements in $W$ and $W_\zeta$, respectively.

Fix a $W$-invariant  non-degenerate bilinear form $(\cdot,\cdot)$ on $\h^\ast$. For each $\alpha\in \Phi_\oa^+$, we set $\alpha^\vee:=2\alpha/(\alpha,\alpha)$.  A weight $\la$ is called {\em integral} if $(\la,\alpha^\vee)\in \Z$ for all roots $\alpha\in \Phi_\oa$. Let $\Lambda$ be the set of integral weights. For any given $\alpha\in \Phi_\oa^+$, a weight $\la$ is called $\alpha$-{\em dominant} (respectively, $\alpha$-{\em anti-dominant}) if $(\la+\rho_\oa,\alpha^\vee)\notin \Z_{<0}$ (respectively,  $(\la+\rho_\oa,\alpha^\vee)\notin \Z_{>0}$). We denote by $\Lnua$ the set of all integral weights that are $\alpha$-anti-dominant, for all $\alpha\in \Pi_\zeta$.

\subsection{Induction and restriction functors}
We denote by $U(\g)$ the universal enveloping algebra of $\g$ and $Z(\g)$ its center. For any $\la\in \h^\ast$, we denote by $\chi_\la: Z(\g)\rightarrow \C$ the central character associated with $\la$. Denote by $\g\mod$ the category of all $\g$-modules.  For any subalgebra $\mf s\subseteq \g$, we have the exact restriction, induction and coinduction functors
\begin{align}
	\Res_{\mf s}^{\g}(-): \g\mod \rightarrow \mf s\mod,~\Ind_{\mf s}^{\g}(-), \Coind_{\mf s}^{\g}(-):  \mf s\mod \rightarrow \g\mod.
\end{align}

In the case $\g_\oa=\mf s$, we use the notations $\Res(-):= \Res_{\g_\oa}^\g(-)$, $\Ind(-):= \Ind_{\g_\oa}^\g(-)$ and $\Coind(-):= \Coind_{\g_\oa}^\g(-)$.  By \cite[Theorem 2.2]{BF} (see also \cite{Go}), $\Ind(-)$ is isomorphic to $\Coind(-)$ up to tensoring with the one-dimensional $\g_\oa$-module  $\wedge^{\dim\g_\ob}(\g_\ob)$.

For any module $M$ having a composition series, we will use $[M : L]$ to denote
the Jordan-H\"older multiplicity of the simple module $L$ in a composition
series of $M$.

\subsection{BGG category $\mc O$} We denote by  $\mc O$ the BGG category  with respect to the triangular decomposition \eqref{eq::tri}. Namely, $\mc O$ consists of finitely-generated $\g$-modules that are semisimple over $\h$ and locally finite over $\mf n^+$. We shall denote the corresponding BGG category of $\g_\oa$-modules by $\mc O_\oa$. For $\la\in \h^\ast$, we denote by $M(\la)$ the Verma module of $\frakb$-highest weight $\la$ and by $L(\la)$ its unique simple quotient.   Denote by $P(\la)$ the projective cover of $L(\la)$ in $\mc O$. We will use the $0$-subscript convention to denote the respective objects in $\mc O_\oa$, e.g., $P_0(\la)$ is the projective cover of $L_0(\la)$ in $\mc O_\oa$.

Denote by $\mc O_\Z$ the integral BGG subcategory, i.e., the full subcategory of $\mc O$ consisting of all modules having integral weight spaces.

We denote by $\mc O(\mf l_\zeta)$ the BGG category of $\mf l_\zeta$-modules   with respect to the Borel subalgebra $\mf l_\zeta\cap \mf b$. Similarly, we shall use the $\mf l_\zeta$-subscript to denote respective $\mf l_\zeta$-modules, e.g., we define the Verma module $M_{\mf l_\zeta}(\la)$ and its simple quotient $L_{\mf l_\zeta}(\la)$ in the category $\mc O(\mf l_\zeta)$. Denote by $P_{\mf l_\zeta}(\la)$ the projective cover of $L_{\mf l_\zeta}(\la)$ in $\mc O(\mf l_\zeta)$. Similarly, we denote by $\mc O(\mf l_\zeta)_\Z$ the integral BGG subcategory of $\mc O(\mf l_\zeta)$.

\subsection{The category $\Opres$} \label{Sect::Opres}

A projective module in $\mc O$ (respectively, $\mc O_\oa$) is called {\em admissible} if any of its simple quotients is of the form $L(\la)$ (respectively, $L_0(\la)$) with $\la\in \Lnua$. We denote by $\Opres$ the full subcategory of $\mc O_\Z$ consisting of modules $M$ that have a two step resolution of the form
\begin{align}
&P_1\rightarrow P_2\rightarrow M \rightarrow 0,
\end{align} where $P_1,P_2$ are admissible projective modules in $\mc O$.
Similarly, we define the full subcategory $\Opres_\oa$ of $\mc O_\oa$ consisting of all  modules that have $2$-step resolutions by direct sums of projective modules of the form $P_0(\la)$ with $\la\in \Lnua$.

Fix $\alpha\in \Pi_\zeta$.  A $\g$-module M is said to be {\em $\alpha$-finite} (respectively, {\em $\alpha$-free}) if the action of non-zero root vectors in  $\g_{-\alpha}$ on $M$ is locally finite (respectively, injective). For $\alpha\in \Pi_\zeta$, we have $(\la+\rho,\alpha^\vee)=(\la+\rho_{\mf l_\zeta},\alpha^\vee)=(\la+\rho_\oa,\alpha^\vee)$,  where $\rho_{\mf l_\zeta}$ denotes the half-sum of positive roots in $\mf l_\zeta$, and hence $L(\la)$ is $\alpha$-free if and only if $L_0(\la)$ is $\alpha$-free if and only if $\la\in \Lnua$, see, e.g., \cite[Lemma 2.1]{CoM1}.

\begin{lem}\label{lem:admissible}
If $P(\la)$ is admissible, then $\Res P(\la)$ is admissible. Conversely, if $P_0(\la)$ is admissible, then $\Ind P_0(\la)$ is admissible. Thus, the functors $\Ind(-)$ and $\Res(-)$ form an adjoint pair of additive functors between the full subcategories  $\Opres_\oa$ and $\Opres$.
\end{lem}

\begin{proof}
Suppose that $P(\la)$ is admissible. We have $0\not=\Hom_{\mc O_\oa}(P_0(\la),\Res L(\la))\cong\Hom_{\mc O}(\Ind P_0(\la),L(\la))$, and hence $P(\la)$ is a direct summand of $\Ind P_0(\la)$. Since $\Res\Ind P_0(\la)$ is admissible, so is $\Res P(\la)$.

Conversely suppose that $P_0(\la)$ is admissible. Now, if $\Hom_{\mc O}(\Ind P_0(\la),L(\mu))\not=0$, then $\Hom_{\mc O_\oa}(P_0(\la),\Res L(\mu))\not=0$. Thus, there is a $\g_\oa$-composition factor in $L(\mu)$ that is $\alpha$-free, for all $\alpha\in\Pi_\zeta$. Since $L(\mu)$ is irreducible, this implies that it is also $\alpha$-free, for all $\alpha\in\Pi_\zeta$ and so $\mu\in\Lambda(\zeta)$. Thus, $\Ind P_0(\la)$ is admissible.
\end{proof}

\subsection{Serre quotient category $\OI$} \label{sect::SerreQ}  Let $\mc I_\zeta$ denote the Serre subcategory of $\mc O_\Z$ generated by simple modules $L(\la)$ with $\la\notin\Lnua$, that is, $L(\la)\in \mc O_\Z$ is $\alpha$-finite for some $\alpha\in \Pi_\zeta$. There is an  associated abelian quotient category  $\OI:=\mc O_\Z/\mc I_\zeta$ in the sense of \cite[Chapter III]{Gabriel} which has the same objects as $\mc O_\Z$ and morphisms
given by
\begin{align}
&\Hom_{\OI}(X,Y):=\lim_{\rightarrow}\Hom_{\mc O}(X',Y/Y'),
\end{align}
where the limit is taken over all $X'\subseteq X$ and $Y'\subseteq Y$ such that $X/X',Y'\in \mc I_\zeta$. There is an exact canonical quotient functor $\pi:\mc O_\Z\rightarrow \OI$, which is the identity on objects and is the canonical map from $\Hom_{\mc O}(X,Y)$ to $\lim\limits_{\rightarrow}\Hom_{\mc O}(X',Y/Y')$. We refer to $\pi$ as the associated {\em Serre quotient functor}.

 By \cite[Lemma 12]{CCM2},  the functor $\pi$ gives rise to an equivalence from $ \mc O^{\vpre}$ to $\ov{\mc O}.$ In particular, $\mc O^{\vpre}$ inherits an abelian category structure and $\{\pi(L(\la))|~\la\in \Lnua\}$  is an exhaustive list of simple objects in $\ov{\mc O}$.  Furthermore, $\pi(P(\la))$ is the indecomposable projective cover of $\pi(L(\la))$ in $\ov{\mc O}$, for any $\la \in \Lnua$.

With respect to the inherited abelian category structure, the induced functors $\Ind(-)$ and $\Res(-)$ between $\Opres_\oa$ and $\Opres$ are exact   since $\Ind(-)$ is also right adjoint to $\Res(-)$, up to an auto-equivalence on $\Opres_\oa$.

\subsection{The Whittaker category $\mc W(\zeta)$} \label{sect::23}

Let $\la\in \h^\ast$. We set $\chi_\la^{\mf l_\zeta}:Z(\mf l_\zeta)\rightarrow \C$ to be the central character associated with $\la$.  We recall from \cite[Section 3.1]{Ch21} the standard Whittaker module is defined as follows: \[M(\la, \zeta):= U(\g)\otimes_{U(\mf p)} Y_\zeta(\la, \zeta),\]
where $Y_\zeta(\la, \zeta):=U(\mf l_\zeta)/(\text{Ker}\chi_\la^{\mf l_\zeta}) U(\mf l_\zeta)\otimes_{U({\mf n^+}\cap \mf l_\zeta)}\mathbb C_\zeta$ denotes Kostant's simple Whittaker modules from \cite{Ko78}. We denote by $L(\la,\zeta)$ the (simple) top of $M(\la, \zeta)$. Furthermore, we have  $${L}(\la, \zeta)\cong  {L}(\mu, \zeta)\Leftrightarrow {M}(\la, \zeta)\cong {M}(\mu, \zeta)\Leftrightarrow W_\zeta \cdot \la =W_\zeta \cdot \mu,$$ for any $\mu \in \h^\ast$. We refer to \cite[Theorem 6]{Ch21} for more details. In particular, we have $M(\la,0)=M(\la)$.
We denote by $\mc N_0(\zeta)$ the corresponding Whittaker category of $\g_\oa$-modules and by $M_0(\la,\zeta)$ the corresponding standard Whittaker module over $\g_\oa$.

Recall that $\mc N(\zeta)$ denotes the category of finitely generated $\g$-modules $M$ that are locally finite over $Z(\g_\oa)$ and $U(\mf n^+)$ such that  $x-\zeta(x)$ acts on $M$ locally nilpotently, for any $x\in \mf n^+_\oa$. The set $\{{L}(\la, \zeta)|\la\in\h^*/W_\zeta\}$ is a complete and non-redundant set of representatives of the isomorphism classes of simple objects in ${\mc N}(\zeta)$.

We consider the category $\mc W(\zeta)$ introduced in \cite[Section 4.4.4]{CCM2}, which is a deformation of the category $\mc O$ and contains the modules $M(\la,\zeta)$ and $L(\la,\zeta)$ for all $\la\in \Lambda$; see also \cite[Theorem 16, Proposition 33]{Ch21}. The category $\mc W(\zeta)$ is defined as the full subcategory of $\mc N(\zeta)$ consisting of modules $M$ that have a two step resolution of modules of the form $E\otimes \Ind(M_0(\nu,\zeta))$, where $E$ is a finite-dimensional $\g$-module and $\nu$ is a dominant (i.e., $\nu$ is $\alpha$-dominant for any $\alpha\in \Phi_\oa^+$) and integral weight such that its stabilizer subgroup of $W$ under the dot-action is equal to $W_\zeta$.

Let us conclude this section with some remarks on equivalences of categories. 	There is an equivalence between $\Opres$ and $\mc W(\zeta)$; see  Section \ref{sect::71}. As a consequence, we have equivalences $\Opres\cong\mc W(\zeta)\cong \OI$. In the case that $\g$ is one from Kac's list \eqref{eq::Kaclist}, a description of objects in $\Opres$ and $\OI$ corresponding to Whittaker modules $M(\la,\zeta)$ and $L(\la,\zeta)$ can be found in Remark \ref{rem::28}.

\section{Proof of Theorem \ref{thm::mainthm}} \label{Sect::3}

In this section, we give a proof of Theorem \ref{thm::mainthm}.

\subsection{Backelin functor $\Gamma_\zeta$} \label{Sect::31}
Recall that the image $\Gamma_\zeta(M)$ of the exact Backelin functor $\Gamma_\zeta:\mc O\rightarrow \mc N(\zeta)$ on a module $M\in \mc O$ is defined as    the submodule consisting of vectors in $\ov M:=\prod_{\la\in \mf h^\ast} M_\la$ annihilated by a power of $x-\zeta(x)$, for all $x\in\mf n^+_\oa$. Let $\Gamma_\zeta^0:\mc O_\oa\rightarrow \mc N_0(\zeta)$ denote the Backelin functor for $\g_\oa$-modules from \cite{B}. Similarly, we denote by $\Gamma_\zeta^{\mf l_\zeta}$ the Backelin functor for $\mf l_\zeta$-modules. We note that $\Res(-)$ intertwines $\Gamma_\zeta$ and $\Gamma^0_\zeta$. The following proposition is crucial.
\begin{prop} \label{prop::3}
	For any $\la\in\h^*$, we have
\begin{align}
&\Gamma_\zeta(M(\la))\cong M(\la,\zeta).
\end{align}
In particular, for any $w\in W_\zeta$, we have
\begin{align*}
\Gamma_\zeta(M(\la))=\Gamma_\zeta(M(w\cdot\la)).
\end{align*}
\end{prop}
\begin{proof}
   By the PBW basis theorem, we may regard $\ov{M(\la)}$ as the vector space $U(\mf u^-_\ob)\otimes \ov{U(\mf u^-_\oa)M_{\mf l_\zeta}(\la)}=U(\mf u^-_\ob)\otimes \ov{M_0(\la)}.$  Therefore,
	$\Gamma_\zeta(M(\la))$ consists of vectors $v$ in 
	$U(\mf u^-_\ob)\otimes \ov{M_0(\la)}$ such that  $x-\zeta(x)$ acts nilpotently on $v$, for each $x\in \mf n^+_\oa$.  There is an embedding $U(\mf u^-_\ob)\otimes \Gamma^0_\zeta(M_0(\la))\hookrightarrow \Gamma_\zeta(M(\la))$ induced by the inclusion $\Gamma^0_\zeta(M_0(\la))\hookrightarrow \ov{M_0(\la)}$. Now by \cite[Proposition 6.9]{B} one has $\Gamma^0_\zeta(M_0(\la))=M_0(\la,\zeta)$, and hence we get an inclusion
	\begin{align} \label{eq::inclu}
	&U(\mf u^-_\ob) X\hookrightarrow \Gamma_\zeta(M(\la)),
	\end{align} such that $U(\mf u^-_\ob)X= U(\mf u^-_\ob)\otimes X$, where $X\subset \ov{M_0(\la)}$ is a $\g_\oa$-submodule isomorphic to $M_0(\la,\zeta)$.
	
	 We claim that the inclusion \eqref{eq::inclu} is an equality. We can argue as follows: For $M\in\mc N(\zeta)$, we have $\Res^\g_{\mf l_\zeta}\Gamma_\zeta(M)=\Gamma^{\mf l_\zeta}_\zeta\Res^\g_{\mf l_\zeta}(M)$. For $\la\in\h^*$ we now compute
	 	\begin{align*}
	 		\Res^\g_{\mf l_\zeta}\Gamma_\zeta(M(\la))=& \Res^\g_{\mf l_\zeta}\Gamma_\zeta(U(\mf u^-)\otimes M_{\mf l_\zeta}(\la)) =  \Res^\g_{\mf l_\zeta}U(\mf u^-)\otimes \Gamma^{\mf l_\zeta}_\zeta(M_{\mf l_\zeta}(\la))\\
	 		=& \Res^\g_{\mf l_\zeta}U(\mf u^-)\otimes M_{\mf l_\zeta}(\la,\zeta) = \Res^\g_{\mf l_\zeta}M(\la,\zeta).
	 	\end{align*}
	 	The third identity above follows from \cite[Lemma 5.12]{MS} and the fact that $U(\mf u^-)$ is a weight module and at the same time a direct sum of finite-dimensional irreducible $\mf l_\zeta$-modules. Thus, we conclude that $U(\mf u^-_\ob)X \hookrightarrow \Gamma_{\zeta}(M(\la))$ is an isomorphism.

We also provide an alternative proof of the subjectivity of the inclusion in \eqref{eq::inclu} as follows. For any $M\in \mc N(\zeta)$ and $c\in \C$, with respect to the operator $H$ from \eqref{eq::paracomp} we denote by $M_c$ the generalized eigenspace of $M$, which lies in $\mc N_{\mf l_\zeta}$ and so has finite length as an $\mf l_\zeta$-module. We note that $U(\mf u^-_\ob)X\hookrightarrow \Gamma_\zeta(M(\la))$ is surjective if and only if the $\mf l_\zeta$-modules $(U(\mf u^-_\ob)X)_c$ and $ \Gamma_\zeta(M(\la))_c$ have the same composition factors for any $c\in \C$.  Recall that $\Res M(\la)$ has a filtration subquotients of which are isomorphic to $M_\oa(\la+\gamma)$ with $\gamma \in \Phi(\wedge(\mf u^-_\ob))$.  Here $\gamma \in \Phi(\wedge(\mf u^-_\ob))$ denotes the set of roots in the space $\wedge(\mf u^-_\ob)$.  Also, for any finite-dimensional weight module $E$ over $\mf l_\zeta$, the tensor product $E\otimes Y_\zeta(\la,\zeta)$ has a filtration subquotients of which are $Y_\zeta(\la+\gamma,\zeta)$ with $\gamma$ being weights of $E$ by \cite[Lemma 5.12]{MS} . Since the restriction functor intertwines $\Gamma_\zeta$ and $\Gamma_\zeta^0$, it follows that the $\mf l_\zeta$-modules $\Gamma_\zeta M(\la)_c$, $(U(\mf u^-_\oa)\otimes U(\mf u^-_\ob)\otimes Y_\zeta(\la,\mf l_\zeta))_c$ and $M(\la,\zeta)_c$ have the same composition factors. As a consequence,  $(U(\mf u^-_\ob)X)_c$ and $\Gamma_\zeta(M(\la))_c$ have the same composition factors for any $c\in \C$.

Recall again $H\in \h$ from the parabolic decomposition \eqref{eq::paracomp}. Since the inclusion in \eqref{eq::inclu} is an equality, it follows that  $H$ acts on $U(\mf u_\ob^-)X=\Gamma_\zeta (M(\la))$ semisimply with the highest eigenvalue $\chi_{\la}^{\mf l_\zeta}(H)$. Let  $V\subseteq \Res_{\mf l_\zeta}^{\g_\oa} X$ isomorphic to $Y_\zeta(\la,\zeta)$, which is the $\chi_\la^{\mf l_\zeta}(H)$-eigenspace in $\Gamma_\zeta (M(\la))$.
Then $\mf u^+ V =0$. Thus, by Frobenius reciprocity we have an isomorphism $M(\la,\zeta)$ to $\Gamma_\zeta(M(\la))$. The conclusion follows.
\end{proof}

\subsection{Proof of Theorem \ref{thm::mainthm}}
For a given $\g$-module $M$, we denote the subspace of {\em Whittaker vectors} associated to $\zeta$ by $Wh_\zeta(M)$, namely,
\begin{align}
&Wh_\zeta(M):=\{m\in M|~xm=\zeta(x)m,\text{ for any $x\in \mf n^+$}\}.
\end{align}

We have the following generalization of \cite[Lemma 37]{BM} to quasi-reductive Lie superalgebras with a similar proof.
\begin{lem}\label{lem:BM}
Let $M$ be either $N(-\mu)^{\ast}$ or $M(\mu)^{\ast,\tau}$. Then $\dim Wh_\zeta(M) =1$.
\end{lem}
\begin{proof}
We observe that $N(-\mu)^*$ as an $U(\mf n^+)$-module is isomorphic to the (full) dual $U(\mf n^+)^*$, which in turn can be identified with the space of linear functions on $U(\mf n^+)$. The $\mf n^+$-character $\zeta$ determines a unique linear function of $U(\mf n^+)$, and hence the Whittaker vectors corresponding to $\zeta$ form a one-dimensional subspace in $N(-\mu)^*$.
The argument for $M(\mu)^{\ast,\tau}$ is analogous.
\end{proof}

A weight is called  $W_\zeta$-anti-dominant if it is $\alpha$-anti-dominant for any root $\alpha$ in $\mf l_\zeta\cap\n^+$. The following theorem completes the proof of Theorem \ref{thm::mainthm} and determines homomorphisms from standard Whittaker modules to the full dual of Verma modules:
\begin{thm} \label{thm::4} For any $\la, \mu \in \h^\ast$, we have
	 \begin{itemize}
	 	\item[(i)] For $M = N(-\mu)^\ast$ or $M(\mu)^{\ast,\tau}$, 	$\Hom_\g(M(\la,\zeta), M)= \left\{\begin{array}{ll}
	 			\C, &  \text{~for } \la\in W_\zeta\cdot \mu;\\
	 			0, & \text{ otherwise};
	 		\end{array} \right. $
	 	\item[(ii)] $\Gamma_\zeta(L(\la))= \left\{\begin{array}{ll}
	 			L(\la,\zeta), &  \text{~if $\la$ is $W_\zeta$-anti-dominant};\\
	 			0, & \text{ otherwise};
	 		\end{array} \right.$
	 	\item[(iii)] $\dim Wh_\zeta(L(\la,\zeta)) =1.$
	 \end{itemize}
\end{thm}
\begin{proof} With Proposition \ref{prop::3} at our disposal, we can now follow the strategy of the proof of \cite[Theorem 20]{Ch21}.   We shall prove Part (i) for the case when $M = N(-\mu)^\ast$.  The case when $M = M(\mu)^{\ast,\tau}$ can be proved in a similar fashion.
	
Since $Wh_\zeta(N(-\mu)^\ast) \cong \C$ by Lemma \ref{lem:BM}, there exists a unique (up to a scalar) Whittaker vector corresponding to $\zeta$ in $N(-\mu)^*$. Now, if $\mu\in W_\zeta\cdot \la$, then we can define a nonzero $\mf l_\zeta$-homomorphism from $M_{\mf l_\zeta}(\la,\zeta)$ to $N(-\mu)^*$ by sending a nonzero Whittaker vector of $M_{\mf l_\zeta}(\la,\zeta)$ to a nonzero scalar multiple of the Whittaker vector of $N(-\mu)^*$ corresponding to $\zeta$. Since $\mf u^+$ annihilates the Whittaker vector in $N(-\mu)^*$, this map induces a $\g$-homomorphism from $M(\la,\zeta)$ to $N(-\mu)^*$. By Lemma \ref{lem:BM} again, it follows that $\Hom_\g(M(\la,\zeta), N(-\mu)^\ast)\cong \C$. On the other hand, if $\mu\not\in W_\zeta\cdot\la$, then there exists $z\in Z(\g_\oa)$ that acts on the generating Whittaker vectors of $M(\la,\zeta)$ and $N(-\mu)^*$ with different scalars. Hence, in this case we have $\Hom_\g(M(\la,\zeta), N(-\mu)^\ast)=0$. This proves Part (i).

	Since $\mu$ is a highest weight of $N(-\mu)^\ast$, we get a nonzero homomorphism $f: M(\mu)\rightarrow N(-\mu)^\ast$. Let $S\subseteq N(-\mu)^\ast$ be the image of $f$. We note that $\ov S\subset N(-\mu)^\ast$ and so $\Gamma_\zeta(S)$ can be regarded as a $\g$-submodule of $N(-\mu)^\ast$. Since $\Gamma_\zeta(-)$ is exact, we get an epimorphism from $M(\mu,\zeta)$ to $\Gamma_\zeta(S)\subset N(-\mu)^\ast$ by Proposition \ref{prop::3}. Now, the module $\Gamma_\zeta(S)$ is generated by a Whittaker vector since it is a quotient of $\Gamma_\zeta(M(\mu))\cong M(\mu,\zeta)$ by Proposition \ref{prop::3}. By Lemma \ref{lem:BM} we have $Wh_\zeta(\Gamma_\zeta(S))\cong \C$ which implies that $\Gamma_\zeta(S)$ is simple, and hence $\Gamma_\zeta(S)\cong L(\mu,\zeta)$.

The exactness of $\Gamma_\zeta$ together with the fact that $\Gamma_\zeta$ sends a quotient of $M(\la)$ to $L(\la,\zeta)$ implies that either $\Gamma_\zeta(L(\la))\cong 0$ or  $\Gamma_\zeta(L(\la))\cong L(\la,\zeta)$. Now, if $\la$ is $W_\zeta$-anti-dominant, then $\Res \Gamma_\zeta(L(\la)) = \Gamma_\zeta^0(\Res L(\la)) \neq 0$ since $\Res L(\la)$ has a composition factor isomorphic to $L_0(\la)$. If $\la$ is not $W_\zeta$-anti-dominant, then there is a $W_\zeta$-anti-dominant weight $\la'\in W_\zeta\cdot\la$ such that  $M_{\mf l_\zeta}(\la')\subseteq M_{\mf l_\zeta}(\la)$. Using parabolic induction we have $M(\la')\subseteq M(\la)$. Also, we have  $\Gamma_\zeta(M(\la')=\Gamma_\zeta(M(\la))$ by Proposition \ref{prop::3}. As $L(\la)$ is a composition factor of $M(\la)/M(\la')$, this implies that $\Gamma_\zeta(L(\la))=0$, and so Part (ii) follows.

Part (iii) follows from the isomorphism $L(\mu,\zeta)\cong \Gamma_\zeta(S)$.	
\end{proof}

\section{Stratified Structure of $\mc O^{\zeta\text{-}\lowercase{pres}}$}\label{Sect::4}
In this section, we show that the category $\Opres$ admits a {\em properly stratified structure} in the sense of \cite{Dl, MaSt04}. This extends the results in \cite[Section 5]{CCM2}, where the case of type I Lie superalgebras was considered.

\subsection{Relation with the category $\mc O(\mf l_\zeta)$} \label{Sect::coappr} Recall that $\mc O(\mf l_\zeta)$ denotes the BGG category of $\mf l_\zeta$-modules with respect to the Borel subalgebra $\mf l_\zeta\cap \mf b$. Also,  $\mc O(\mf l_\zeta)_\Z$ denotes the integral BGG subcategory of $\mf l_\zeta$-modules. In this subsection, we provide a characterization of objects in $\Opres$ in terms of  projective-injective modules in $\mc O(\mf l_\zeta)_\Z$. In \cite[Theorem 2]{KoM}, K\"onig and Mazorchuk have shown that objects in $\Opres_\oa$ are modules in $\mc O_\oa$ which, as $\mf l_\zeta$-modules, are presented by projective-injective modules in $\mc O(\mf l_\zeta)_\Z$. The following is a generalization to Lie superalgebras.
\begin{prop}\label{lem::CharaOpres}
	 Let $M\in \mc O_\Z$. Then the following are equivalent:
	\begin{itemize}
		\item[(1)] $M\in \mc O^{\vpre}$.
		\item[(2)] $\Res M\in \mc O_\oa^{\vpre}$.
		\item[(3)] $\Res_{\mf l_\zeta}^\g M$ is a direct sum of modules having a two step presentation by projective-injective modules in $\mc O(\mf l_\zeta)$.
	\end{itemize}
\end{prop}
Before we prove Proposition \ref{lem::CharaOpres}, we need some preparatory results.  Recall that a projective module $M\in \mc O$ is said to be admissible if $M$ is a direct sum of projective modules of the form $P(\la)$ with $\la\in \Lnua$. For any $M\in \mc O$, we let $Tr(M)$ and $Tr_0(\Res M)$ denote the sum of images of homomorphisms from admissible projective modules in $\mc O$ (respectively, in $\mc O_\oa$) to $M$ (respectively, to $\Res M$).
Recall the {\em coapproximation functor} $\mf j_0:\mc O_\oa\rightarrow \Opres_\oa$ with respect to admissible projective modules studied in \cite[Section 2.4]{MaSt04}; see also \cite[Section 3]{Au}. In general, we define the    coapproximation functor $\mf j:\mc O\rightarrow \Opres$ for Lie superalgebras in the same fashion. Let $M\in \mc O$ with an epimorphism  $f:P_M\onto Tr(M)$ from a projective module $P_M\in \mc O$. Then, $\mf j(M)$ is defined to be the quotient module  $P_M/Tr(\text{ker}(f)).$ This definition is independent of the choices made. We refer to \cite[Section 3]{Au} and \cite[Section 2.4]{KM2}  for more details.

We need the following auxiliary lemma.
\begin{lem} \label{lem::commjj0}
	We have $\Res \mf j(M)\cong \mf j_0\Res(M)$, for any $M\in \mc O$.
\end{lem}
\begin{proof}
	It suffices to show that $\Res Tr(M)=Tr_0(\Res M)$ for any $M\in \mc O$. 	To see this, let $\phi: Q\rightarrow \Res M$ be a homomorphism from an admissible projective module $Q\in\mc O_\oa$ to $\Res M$. This gives rise to $\Ind\phi:\Ind Q\rightarrow M$ with $\Ind Q$ admissible in $\mc O$ by Lemma \ref{lem:admissible}. Now, the image of $\phi$ is contained in the image of $\Res\Ind(\phi): \Res \Ind Q\rightarrow \Res M$ since $ Q$ is a direct summand of the admissible projective module $\Res \Ind Q$, which shows that $Tr(M)\supseteq Tr_0(\Res M).$ Conversely, let $\psi: P\rightarrow M$ be a homomorphism from an admissible projective module $P\in \mc O$. Since the image of $\psi$ is equal to that of $\Res(\psi)$ and $\Res P$ is admissible in $\mc O_\oa$ by Lemma \ref{lem:admissible} again, it follows that they are contained in $Tr_0(\Res M)$.  	This completes the proof.
\end{proof}

\begin{proof}[Proof of Proposition \ref{lem::CharaOpres}]
The equivalence between $(2)$ and $(3)$ is a consequence of \cite[Theorem 2]{KoM}. Also, $(1)\Rightarrow (2)$ is a direct consequence of Lemma \ref{lem:admissible}. So it remains to show that $(2)\Rightarrow (1)$. Since $\Res M\in \Opres_\oa$, it follows that every simple quotient of $M$ is $
 \alpha$-free for any $\alpha \in \Pi_\zeta$. This implies that the top of $M$ is a direct sum of modules of the form $L(\la)$ with $\la \in \Lnua$. Therefore, there is an epimorphism $f: P\onto M$ for some  admissible projective module $P\in \mc O$. Consequently, we obtain an epimorphism from $\mf j(M)$ to  $M$.

 On the other hand,  we note that $\Res M\in \Opres_\oa$ if and only if $\mf j_0(\Res M)\cong \Res M$. It follows that $\Res \mf j(M)\cong \Res M$ by Lemma \ref{lem::commjj0}. Consequently, the epimorphism from $\mf j(M)$ to $M$ is an isomorphism, as desired.
 \end{proof}

\begin{rem}
	We remark that there is an analogue of Lemma \ref{lem::commjj0} for the full subcategory of injectively copresentable modules of $\mc O$ in \cite[Theorem 10]{Ch212}, which is proved by means of Deodhar-Mathieu's version of Enright completion functors. In fact, these functors are isomorphic to certain {\em approximation functors}, which are the dual version of the coapproximation functors introduced above; see also \cite[Section 2, Theorem 2]{KoM}, \cite[Theorem 5]{KM2}.
\end{rem}

\subsection{Standard and proper standard objects.}
 Recall that $P(\la)$ and $P_0(\la)$ denote the projective covers of $L(\la)$ and $L_0(\la)$ in $\mc O$ and $\mc O_\oa$, respectively. For $\la\in\Lambda(\zeta)$, we define the {\em proper standard} object $\ov \Delta(\la)$ for $\mc O^{\vpre}$ to be $P(\la)/ Tr(K_\la)$, where $K_\la$ is the kernel of the canonical quotient $P(\la)\rightarrow M(\la)$. Note that the proper standard objects indeed lie in $\mc O^{\vpre}$  and furthermore
\begin{align}\label{eqn:aux13}
[\ov{\Delta}(\la)/M(\la):L(\gamma)]=0, \text{ for } \gamma\in\Lnua.
\end{align}
We also define the {\em standard} object $\Delta(\la)$ for $\mc O^{\vpre}$ by
\begin{align}
	&\Delta(\la):= \Ind_{\mf p}^{\g} P_{\mf l_\zeta}(\la).
	\end{align}
Here, $P_{\frakl_\zeta}(\la)$ is regarded as a $\frakp$-module by letting $\frak u^+$ act trivially. Similarly, we define $\Delta_0(\la):= \Ind_{\mf p_\oa}^{\g_\oa} P_{\mf l_\zeta}(\la). $
It follows from   Proposition \ref{lem::CharaOpres}  that $\Delta(\la)$ is an object in $\Opres$. We will provide an alternative proof in Corollary \ref{coro::standardob} below.

We set $\mc F(\Delta)$ to be the full subcategory of $\mc O$ consisting of modules which admit a $\Delta$-flag. Similarly, we define $\mc F(\Delta_0)\subset \mc O_\oa$.

\begin{lem} \label{lem::indDelta}
	For any $V\in \mc F(\Delta_0)$, we have $\Ind(V)\in \mc F(\Delta)$.
\end{lem}
\begin{proof}
	It is sufficient to prove this for $V=\Delta_0(\la)$ with $\la\in \Lnua$. We note that \begin{align}
	&\Ind\Delta_0(\la)\cong \Ind^{\g}_{\g_\oa}\Ind^{\mf \g_\oa}_{\mf p_\oa}P_{\mf l_\zeta}(\la)\cong \Ind^{\g}_{\mf p}\Ind^{\mf p}_{\mf p_\oa}P_{\mf l_\zeta}(\la).
	\end{align}
 We claim that the $\mf p$-module $\Ind^{\mf p}_{\mf p_\oa}P_{\mf l_\zeta}(\la)$ has a filtration any subquotient $S$ of which is of the form $\Res_{\mf l_\zeta}^{\mf p}S\cong P_{\mf l_\zeta}(\gamma)$ with $\gamma \in \Lnua$ and $\mf u^+S=0$. To see this, we recall the grading operator $H\in \mf h$ introduced in \eqref{eq::paracomp}.

 Set $E:=U(\mf u^+_\ob)$, which is an $\mf l_\zeta$-submodule of $U(\mf g)$ under the adjoint action.
 For each complex number $c\in \C$, we let $E_c$ denote the $c$-eigenspace of $H$ under the adjoint action. Since $H$ lies in the center of $\mf l_\zeta$, it follows that each $E_c$ is an $\mf l_\zeta$-submodule of $E$. Then, $E$ decomposes into a sum of $\mf l_\zeta$-submodules
 \begin{align}
 &E=\oplus_{i=1}^kE_{c_i},
 \end{align} for some $c_1,\ldots,c_k\in \C$ such that $\text{Re}(c_1)<\text{Re}(c_2)<\cdots<\text{Re}(c_k)$.

 We may note that \begin{align}
&\Res_{\mf l_\zeta}^{\mf p}\Ind_{\mf p_\oa}^{\mf p}P_{\mf l_\zeta}(\la)\cong E\otimes P_{\mf l_\zeta}(\la) = \bigoplus_{i=1}^k (E_{c_i} \otimes P_{\mf l_\zeta}(\la)),
\end{align} is a direct sum of projective-injective $\mf l_\zeta$-modules in the category $\mc O(\mf l_\zeta)$, each of which is of the form $P_{\mf l_\zeta}(\gamma)$ with $\gamma\in \Lnua$. Furthermore, we define a filtration
\begin{align}\label{eq::45}
&0=\Ind_{\mf p_\oa}^{\mf p}P_{\mf l_\zeta}(\la)_{k+1}\subset \Ind_{\mf p_\oa}^{\mf p}P_{\mf l_\zeta}(\la)_k \subset \Ind_{\mf p_\oa}^{\mf p}P_{\mf l_\zeta}(\la)_{k-1}\subset \cdots \subset \Ind_{\mf p_\oa}^{\mf p}P_{\mf l_\zeta}(\la)_1 = \Ind_{\mf p_\oa}^{\mf p}P_{\mf l_\zeta}(\la).
\end{align} by letting
\begin{align}
&\Ind_{\mf p_\oa}^{\mf p}P_{\mf l_\zeta}(\la)_j:= \bigoplus_{i=j}^k (E_{c_i} \otimes P_{\mf l_\zeta}(\la)),
\end{align} for $j=1,\ldots, k$. Fix $1\leq j\leq k$. We note that both $\mf u_\ob^+E_{c_j}$ and $[\mf u_\oa^+, E_{c_j}]$ are subspaces of $ \bigoplus_{i>j}E_{c_i}.$ Consequently, we have
\begin{align}
	&\mf u^+(\Ind_{\mf p_\oa}^{\mf p}P_{\mf l_\zeta}(\la))_j \subset (\Ind_{\mf p_\oa}^{\mf p}P_{\mf l_\zeta}(\la))_{j+1}.
	\end{align}
Therefore, the filtration \eqref{eq::45} gives rise to the desired filtration. This completes the proof.
\end{proof}

\begin{lem}  \label{lem::8}
	Suppose that $M\in \mc O$ and $c\in \C$. Let $M_c$ be  the $c$-eigenspace with respect to the action of the grading operator $H$ from \eqref{eq::paracomp}. Then we have
	\begin{itemize}
		\item[(1)] There is $d\in \C$ such that $M_{d} \neq 0$ and $M_c=0$ for  $\text{Re}(c)>\text{Re}(d).$
		\item[(2)]  $M_c$ is an object in $\mc O(\mf l_\zeta)$.
		\item[(3)] If $M$ is a (possibly infinite) direct sum of projective-injective modules in $\mc O(\mf l_\zeta)$, then so is $M_c$.
	\end{itemize}
\end{lem}
\begin{proof} Since $H$ acts on any $\mf h$-weight subspace of $M$ as a scalar and $M$ is finitely generated as an $\mf n^-$-module, the conclusion in Part (1) follows.

	For the proof of Part (2), we  note that $M_c$ is an $\mf l_\zeta$-module, locally finite over $ \mf l_\zeta\cap \mf b$ and semisimple over $\mf h$. Now, if $M$ is  simple, then $M$ is an epimorphic image of $\Ind_{\mf p_\oa}^{\mf g}L_{\mf l_\zeta}(\la)$, for some $\la\in\frakh^\ast$. Since $\Ind_{\mf p_\oa}^{\mf g}L_{\mf l_\zeta}(\la)_c$ is finitely generated, so is $M_c$. Consequently, $M_c \in \mc O(\mf l_\zeta)$ by induction on the length of the module $M$ in $\mc O$.
	
	Finally, suppose that $\Res_{\mf l_\zeta}^{\g}M= \bigoplus_{\la \in \mc X}X_\la$, where $X_\la\cong \plz(\la)$ are projective-injective modules in $\mc O(\mf l_\zeta)$ for some $\la$ lying in some index set $\mc X$. Since each $X_\la$ is indecomposable and $H$ acts  on $X_\la$ semisimply, it follows that $(H-\chi_{\la}^{\mf l_\zeta}(H))X_\la=0.$ Therefore, $M_c$ is projective-injective in $\mc O(\mf l_\zeta)$. This completes the proof.
\end{proof}

\begin{lem} \label{lem::9}
	Let $M\in \mc F(\Delta)$. Then we have
	\begin{itemize}
		\item[(1)] $\Res_{\mf l_\zeta}^{\g}M$ is a direct sum of projective-injective $\mf l_\zeta$-modules in $\mc O(\mf l_\zeta).$
    \item[(2)]	$\Res M\in \mc F(\Delta_0)$.
	\end{itemize}
\end{lem}
\begin{proof}
It suffices to prove the assertions for $M= \Delta(\la)$, for a given $\la \in \Lnua$.

We have $\Res_{\mf l_\zeta}^{\g}\Delta(\la)\cong U(\mf u^-)\otimes \plz(\la)$, and hence $\Res_{\mf l_\zeta}^{\g}\Delta(\la)$ is a direct sum of projective-injective $\mf l_\zeta$-modules in $\mc O(\mf l_\zeta)$. This proves Part (1).

Since $\Delta(\la)$ has a $\frakg_\oa$-Verma flag as well,  Part (2) now follows from \cite[Proposition 2.13(i),(ii)]{MaSt04}.
\end{proof}

The following is a generalization of \cite[Corollary 2.14]{MaSt04}, where the case of Lie algebras  was considered. Here we provide an alternative proof.
\begin{prop} \label{prop::FDsummand}
	The full subcategory $\mc F(\Delta)$ is closed under taking direct summands.
\end{prop}
\begin{proof}
 Let $M\in \mc F(\Delta)$. Set $N$ to be a direct summand of $M$. We shall show that $N\in \mc F(\Delta)$. Let
 \begin{align}
 &0=M^{k+1}\subset M^k\subset M^{k-1}\subset \cdots\subset M^1 =M \label{eq::48}
 \end{align} be a $\Delta$-flag.

 We decompose $\Res^{\g}_{\mf l_\zeta} M$ into a direct sum
\begin{align}
&M=\bigoplus_{c\in \C} M_c,
\end{align} where $M_c$ is the $c$-eigenspace with respect to the action of the grading operator $H$ from \eqref{eq::paracomp}. By Lemmas \ref{lem::8} and \ref{lem::9}, each $M_c$ is a projective-injective module in $\mc O(\mf l_\zeta)$. Also, there is a complex number $d\in \C$ such that $M_d\neq 0$ and $M_c=0$ for any $c\in \C$ with $\text{Re}(c)>\text{Re}(d)$.

Note that  $U(\g)V=U(\mf u^-)V$, for any direct summand  $V$ of $M_d$. We claim that both $U(\g)V$ and  $M/(U(\g)V)$ lie in $\mc F(\Delta)$. We shall prove by induction on the length $k$. The conclusion for the case $k=1$ holds since $U(\g)M_d=M$ in this case.

 Since $\eqref{eq::48}$ is a $\Delta$-flag, there is  $1\leq   i< k+1$ with the canonical quotient $p: M^{i}\onto M^{i}/M^{i+1}$ such that
 \begin{align}
 	&(M^{i}/M^{i+1})_d\cong (\Ind_{\mf p}^\g\plz(\la))_d\cong \plz(\la),
 \end{align}  for some $\la\in \Lnua.$  Consider the following split short exact sequence induced by $p$  and the inclusion $M^{i+1}\hookrightarrow M^i$
\[0\rightarrow (M^{i+1})_d\rightarrow (M^i)_d \rightarrow (M^{i}/M^{i+1})_d\rightarrow 0\]
in $\mc O(\mf l_\zeta)$.  We get an $\mf l_\zeta$-submodule $S$ of $M_d$ such that $p|_S: S\rightarrow \plz(\la)$ is an isomorphism of $\mf  l_\zeta$-modules. Since $U(\g)S=U(\mf u^-)S$, as observed above, it follows that $p$ restricts to a $\g$-module isomorphism from $U(\mf g)S$ to $M^{i}/M^{i+1}\cong \Ind_{\mf p}^\g\plz(\la).$ Consequently, we get a split short exact sequence
\begin{align}
&0\rightarrow M^{i+1}\rightarrow M^i \rightarrow M^{i}/M^{i+1}\rightarrow 0
\end{align} in $\mc O$.

Put $K:=U(\g)S$, and so $M^{i}=K\oplus M^{i+1}$. Consider the exact sequence
\begin{align}
	&0\rightarrow M^{i+1}\rightarrow M/K \rightarrow M/M^{i}\rightarrow 0.
\end{align} Since both $M^{i+1}$ and $M/M^i$ lie in $\mc F(\Delta)$, it follows that $M/K$ is an object in $\mc F(\Delta)$. By induction it follows that $U(\mf g)M_d=U(\mf u^-)M_d\cong \Ind_{\mf p}^\g M_d$ is a direct sum of standard objects $\Delta(\mu)$, for $\mu\in \Lnua$, and $M/(U(\g)M_d)$ lies in $\mc F(\Delta)$. Let $V$ be a direct summand of $M_d$, then $U(\g)V\cong \Ind_{\mf p}^\g V\in \mc F(\Delta)$ and $U(\mf g)(M/V)\cong \Ind_{\mf p}^\g (M/V)\in \mc F(\Delta)$. Consequently, we have $M/(U(\g)V)\in \mc F(\Delta)$, as desired.

Finally, let $M=N\oplus N'$. Then, either $N_d\neq 0$ or $N'_d\neq 0$. Without loss of generality, we may assume that $N_d\neq 0$, then  $M/(U(\g)N_d) = (N/U(\g)N_d)\oplus N'$ has a $\Delta$-flag by the argument above. By induction on the length of $M$, it follows that $N/(U(\g)N_d)$ and $N'$ have $\Delta$-flags, and hence $N$ as well. This completes the proof.
\end{proof}

\subsection{Stratified structure and BGG reciprocity}
 We denote by $\leq$ the partial order on $\mf h^\ast$ induced by the triangular decomposition \eqref{eq::tri}. Namely, it is the transitive closure of the relations
	$$\begin{cases}
	\lambda-\alpha \le\lambda, &\mbox{for a positive root $\alpha$},\\
	\lambda+\alpha \le\lambda, &\mbox{for a negative root $\alpha$}.
	\end{cases}$$

Recall the Serre quotient category $\OI$ and the canonical quotient functor $\pi:\mc O_\Z\rightarrow\OI$  associated to $\zeta$ from Section \ref{sect::SerreQ}.
Lemma \ref{lem::indDelta}, combined with Proposition \ref{prop::FDsummand}, has the following consequence. Namely, the category $\mc O^{\vpre}$ is stratified.
	\begin{thm} \label{thm::str}
		For any $\la \in \Lnua$, we have the following.
		\begin{itemize}
			\item[(1)] There is a short exact sequence in $\mc O$ of the form
			\begin{align}
			&0\rightarrow K(\la)\rightarrow P(\la)\rightarrow \Delta(\la)\rightarrow 0, \label{eq::413}
			\end{align} such that $K(\la)$ has a $\Delta$-flag, subquotients of which are isomorphic to $\Delta(\mu)$ with $\la<\mu$.
			\item[(2)] The object $\pi\ov \Delta(\la)$ surjects onto $\pi L(\la)$, and its kernel has a composition series, subquotients of which are isomorphic to $\pi L(\mu)$ with $\mu\in \Lnua$ and $\mu<\la$.
			\item[(3)] The object $\pi(\Delta(\la))$ has a filtration, subquotients of which are isomorphic to $\ov \Delta(\la)$.
		\end{itemize}
	\end{thm}
\begin{proof}
	Since $\Res P(\la)$ is admissible by Lemma \ref{lem:admissible}, it follows that $\Res P(\la)\in \mc F(\Delta_0)$ by \cite[Lemma 4]{KoM} and \cite[Propositions 2.9 and 2.13]{MaSt04}. Alternatively, this also follows from the fact that $\Res P(\la)$ is a direct summand of a sum of translations of a module of the form $\Ind_{\mf p_\oa}^{\mf g_\oa}P_{\mf l_\zeta}(\gamma)$ with $\gamma\in \Lnua$; see also the proof of \cite[Lemma 4]{KoM}.	 By Lemma \ref{lem::indDelta}, we have $\Ind \Res P(\la)\in \mc F(\Delta)$. Since $P(\la)$ is a direct summand of $\Ind \Res P(\la)$, Part (1) follows from Proposition \ref{prop::FDsummand} and the BGG reciprocity for $\mc O$.
	
The exactness of $\pi$ and \eqref{eqn:aux13} imply that $\pi{\ov\Delta}(\la) =\pi(M(\la))$, and hence Part (2) follows.

	Finally, we prove Part (3). We note that, for a given $\gamma\in \Lambda$ with $\gamma\notin \Lnua$, we have  $\pi(\Ind_{\mf p}^{\g}L_{\mf l_\zeta}(\gamma))=0$. The conclusion of Part (3) follows. This completes the proof.
\end{proof}

Set $Q(\la)$ to be the quotient of $P(\la)$ modulo the sum of the images of all homomorphisms $P(\mu)\rightarrow P(\la)$, where $\mu \in \Lnua$ and $\mu>\la$.  Observe that $Q(\la)\in \mc O^{\vpre}.$

\begin{cor} \label{coro::standardob} For any $\la \in \Lnua$, we have  $\Delta(\la)\cong Q(\la)$.
\end{cor}
\begin{proof}
 Let $K(\la)\in \mc F(\Delta)$ be the module given in Theorem \ref{thm::str}(1). Then $K(\la)$ has a $\Delta$-flag, subquotients of which are $\Delta(\mu)$ with $\mu\in \Lnua$ and $\mu>\la$. Therefore, $K(\la)$ is the image of a homomorphism from an admissible projective module, and hence  we have an epimorphism $\Delta(\la) \twoheadrightarrow Q(\la)$. Now, to complete the proof, we note that, for any $\mu\in \Lnua$ with $\mu>\la$, $\Hom_{\mc O}(P(\mu), \Delta(\la))=[\Delta(\la):L(\mu)]=0$ by Theorem \ref{thm::str}.
\end{proof}

For any $M\in \mc F(\Delta)$, we denote by $(M : \Delta(\la))$ the  multiplicity of a standard object $\Delta(\la)$ in a $\Delta$-flag of $M$. The following corollary determines the multiplicities of   standard objects of the projective cover in $\mc O^{\vpre}.$

\begin{cor}[BGG reciprocity]\label{cor:BGG}
  For any $\la,\mu\in \Lnua$, we have
  \begin{align*}
  &( P(\la):\Delta(\mu)) = [\ov \Delta(\mu): L(\la)] = [M(\mu):L(\la)].
  \end{align*}
\end{cor}
 \begin{proof}
 	Let $\gamma\in \h^\ast$ be an integral weight with $\gamma\not\in \Lnua$. Then $\Ind_{\mf p}^{\g} L_{\mf l_\zeta}(\gamma)$ is $\alpha$-finite for some $\alpha\in \Pi_\zeta$, which implies that $[\Ind_{\mf p}^{\g} L_{\mf l_\zeta}(\gamma): L(\la)]=0$ since $L(\la)$ is $\alpha$-free; see also Section  \ref{Sect::Opres}. Therefore, for any $w\in W_\zeta$ we have $[M(w\cdot\mu)/M(\mu):L(\la)]=0$ and hence
 	\begin{align}\label{eqn:aux12}
 		& [M(w\cdot \mu): L(\la)] = [M(\mu):L(\la)].
 	\end{align}  By BGG reciprocity for category $\mc O$, we have $(P(\la): M(w\cdot \mu))=(P(\la):M(\mu))$. Finally, noting that the multiplicity of $\overline \Delta(\mu)$ in a $\ov \Delta$-flag of $\Delta(\mu)$ is $|W_\zeta\cdot \mu|$, the conclusion follows.
 \end{proof}

\subsection{Tilting modules} 
In this subsection, we assume that $\g$ is a basic Lie superalgebra, i.e., $\g$ is one from \eqref{eq::Kaclist} with $\g\neq \mf p(n)$. There is a simple-preserving duality $D$ on $\mc O$, i.e., $D$ is a contravariant auto-equivalence  such that $D^2\cong \Id$ and $D(L(\la))\cong L(\la)$, for each $\la \in \mf h^\ast$. 

For $\la\in \Lnua$, we define the {\em co-standard} object $\nabla(\la)$ for $\mc O^{\vpre}$ by
\begin{align}
	&\nabla(\la):= D \Ind_{\mf p}^{\g} P_{\mf l_\zeta}(\la).
\end{align}
It follows from Proposition \ref{lem::CharaOpres} that $\nabla(\la)\in \Opres$; see also \cite[Corollary 2.11]{MaSt04}. Similarly, we define co-standard object $\nabla_0(\la)$  for $\Opres_\oa$.  If $\mc O^{\vpre}=\mc O$, then we arrive at the usual definition of dual Verma modules.

 A module $M\in \Opres$ is referred to as a {\em tilting module} if $M$ has both a $\Delta$-flag and a $\nabla$-flag. In the case that $\mc O^{\vpre}=\mc O$, for each $\la\in \mf h^\ast$ there is a unique indecomposable tilting module of highest weight $\la$, which we denote by $T^{\mc O}(\la)$, namely, $T^{\mc O}(\la)$ has both a Verma flag and a dual Verma flag; see, e.g., \cite{Br04}. We also denote by $T^{\mc O_\oa}(\la)$ the indecomposable tilting module of highest weight $\la$ in $\mc O_\oa$.

In the case that $\mf g=\mf g_\oa$, the classification of indecomposable  tilting modules in $\Opres_\oa$ is given in \cite[Section 6]{FKM00}; see also \cite[Section 5.4]{MaSt04}. In particular, the list  $\{T^{\mc O_\oa}(w_0^\zeta\cdot \la)|~\la\in \Lnua\}$ forms a complete set of non-isomorphic indecomposable tilting modules  in $\Opres_\oa$, and each $T_0(\la):= T^{\mc O_\oa}(w_0^\zeta\cdot \la)$ is the unique indecomposable tilting module containing $\Delta_0(\la)$  such that $T_0(\la)/\Delta_0(\la)$ has a $\Delta_0$-flag. Here we recall that $w_0^\zeta$ denotes the longest element in $W_\zeta$.  We set
	\begin{align}
		&T(\la):=T^{\mc O}(w_0^\zeta\cdot\la),
	\end{align} for each $\la\in \Lnua$. Below we gives a classification of indecomposable tilting objects in $\mc O^{\vpre}$. We remark that the type I case was considered in \cite[Section 6.1]{CCM2}.

\begin{prop}\label{prop:tilting}
	The modules $\{T(\la)|~\la\in \Lnua\}$ constitute
an exhaustive list of indecomposable tilting modules in $\Opres$.  In particular, each $T(\la)$ is the unique indecomposable tilting module in $\mc O^{\vpre}$ containing $\Delta(\la)$ such that $T(\la)/\Delta(\la)$ has a $\Delta$-flag.
\end{prop}
\begin{proof}
Let $\la\in\Lambda(\zeta)$ and set $\mu=\la-2\rho_1\in \Lambda(\zeta)$, where $\rho_1$ denotes the half-sum of roots in $\mf n^+_\ob$.  We have
\begin{equation*}
\Ind\Delta_0(\mu)=\Ind_{\g_\oa}^{\g}\Ind_{\mf p_\oa}^{\g_\oa} P_{\mf l_\zeta}(\mu) \cong
\Ind_{\mf p}^{\g}\Ind_{\mf p_\oa}^{\mf p} P_{\mf l_\zeta}(\mu) \cong
\Ind_{\mf p}^{\g}\left(U(\mf n^+_\ob)\otimes P_{\mf l_\zeta}(\mu)\right).
\end{equation*}
Now, $P_{\mf l_\zeta}(\mu)$ has an ${\mf l_\zeta}$-Verma flag with highest term $M_{\mf l_\zeta}(w_0^\zeta\cdot\mu)$, and hence $U(\mf n^+_\ob)\otimes P_{\mf l_\zeta}(\mu)$ has a ${\mf l_\zeta}$-Verma flag  with highest term $M_{\mf l_\zeta}(w_0^\zeta\cdot\mu+2\rho_{1})=M_{\mf l_\zeta}(w_0^\zeta\cdot\left(\mu+2\rho_{1}\right))=M_{\mf l_\zeta}(w_0^\zeta\cdot\la)$. Therefore, $\Ind\Delta_0(\mu)$ has a Verma flag with highest term $M(w_0^\zeta\cdot\la)$. This implies that $\Ind T_0(\mu)$ has a Verma flag with highest term $M(w_0^\zeta\cdot\la)$, and hence by Lemma \ref{lem::indDelta}, a $\Delta$-flag with $\Delta(\la)$ as a highest term.  Therefore,  the indecomposable direct summand of $\Ind T_0(\mu)$ containing $\Delta(\la)$ is (isomorphic to) $T(\la)$. By Proposition \ref{prop::FDsummand}, $T(\la)$ has a $\Delta$-flag. Also, by Lemma \ref{lem:admissible}, $\Ind T_0(\mu)\in\Opres$ and therefore, $T(\la)$ lies in $\Opres$. Since $T(\la)$ is self-dual, it follows that $T(\la)$ has both a $\Delta$-flag and a $\nabla$-flag. Therefore, $T(\la)$ is an indecomposable tilting module in $\Opres.$

Now, if $T$ is an indecomposable tilting module in $\Opres$ with a $\Delta$-flag starting at $\Delta(\la)$, then by uniqueness of tilting module (see for example \cite[Proposition 5.6]{So98}), we have $T\cong T(\la)$. This completes the proof.
\end{proof}

\begin{cor}\label{cor:ringel:dual}(Ringel duality)
For $\la,\mu\in\Lnua$ we have
\begin{align}\label{eqn:aux11}
(T(\la):\Delta(\mu)) = [\ov{\Delta}(-w_0^\zeta\cdot\mu): L(-w_0^\zeta\cdot\la)] = (P(-w_0^\zeta\cdot\la):\Delta(-w_0^\zeta\cdot\mu)).
\end{align}
\end{cor}

\begin{proof}
Let $\la,\mu\in\Lnua$. By \eqref{eqn:aux12}, the BGG reciprocity and the Ringel-Soergel duality in $\mc O$ \cite{Br04} (see also \cite[Theorem 3.7]{CCC}), we have
\begin{align*}
(T^{\mc O}(w_0^\zeta\cdot\la):M(\mu))= (T^{\mc O}(w_0^\zeta\cdot\la):M(w\cdot\mu)),\quad\forall w\in W_\zeta.
\end{align*}
Thus, it follows from Proposition \ref{prop:tilting} that
\begin{align*}
(T(\la):\Delta(\mu))=&(T^{\mc O}(w_0^\zeta\cdot\la):M(w_0^\zeta\cdot\mu))\\
=&[M(-w_0^\zeta\cdot\mu):L(-w_0^\zeta\cdot\la)] = [\ov{\Delta}(-w_0^\zeta\cdot\mu): L(-w_0^\zeta\cdot\la)].
\end{align*}
The second identity in \eqref{eqn:aux11} follows from Corollary \ref{cor:BGG}.
\end{proof}

\section{$\imath$-Canonical Bases on Symmetrized Fock Space} \label{Sect::QSPcan}

In this section we introduce a quasi-split quantum symmetric pair $(U,U^\imath)$ of type $AIII$ and study a $q$-symmetrized Fock space which carries an action of the $\imath$-quantum group $U^\imath$. We construct $\imath$-canonical and dual $\imath$-canonical bases on this space. In Section \ref{Sect::thm:match:can} we shall relate these bases in a natural way to tilting and simple objects, respectively, in the category $\mc W(\zeta)$ in the case when $\frakg$ is an ortho-symplectic Lie superalgebra.

\subsection{Quantum symmetric pair $(U,U^\imath)$}\label{sec:QSP:def}

Let $U$ denote the infinite-rank quantum group of type $A$ with Chevalley generators $\{E_{\alpha_i},F_{\alpha_i},K^{\pm1}_{\alpha_i}|i\in\Z\}$, i.e., $U$ is the associative algebra over $\bbQ(q)$ generated by these generators subject to certain well-known relations (see, e.g., \cite[Section 2.1]{CLW2}). The anti-linear ($q\rightarrow q^{-1}$) map $\psi$ preserving $E_{\alpha_i}$ and $F_{\alpha_i}$ and sending $K_{\alpha_i}$ to $K^{-1}_{\alpha_i}$ induces the well-known bar involution on $U$. The algebra $U$ is a Hopf algebra with co-multiplication $\Delta$ \cite{Kas}:
\begin{align*}
&\Delta(E_{\alpha_i})=1\otimes E_{\alpha_i}+E_{\alpha_i}\otimes K^{-1}_{\alpha_i},\\
&\Delta(F_{\alpha_i})=F_{\alpha_i}\otimes 1+K_{\alpha_i}\otimes F_{\alpha_i},\\
&\Delta(K_{\alpha_i})=K_{\alpha_i}\otimes K_{\alpha_i}.
\end{align*}

Now let $U^\imath$ be the associative algebra over $\bbQ(q)$ generated by $\{e_{\alpha_i},f_{\alpha_i},k^{\pm 1}_{\alpha_i},t|i>0\}$ subject to the relations in \cite[Chapter 2.1]{BaoWang}. Then, the anti-linear map $\psi^\imath$ that preserves the generators $e_{\alpha_i},f_{\alpha_i},t$, and that sends $k_{\alpha_i}$ to $k^{-1}_{\alpha_i}$, for $i>0$, defines an anti-linear involution on $U^\imath$.

According to \cite[Proposition 2.2]{BaoWang}, there is an embedding of associative algebras $\imath:U^\imath\rightarrow U$ given by
\begin{align*}
\imath(k_{\alpha_i})=K_{\alpha_i}K^{-1}_{\alpha_{-i}},\quad \imath(t)=E_{\alpha_0}+qF_{\alpha_0}K^{-1}_{\alpha_0}+K^{-1}_{\alpha_0},\\
\imath(e_{\alpha_i})=E_{\alpha_i}+K^{-1}_{\alpha_i}F_{\alpha_{-i}},\quad \imath(f_{\alpha_i})=F_{\alpha_i}K^{-1}_{\alpha_{-i}}+E_{\alpha_{-i}},
\end{align*}
so that $U^\imath$ may be regarded as a subalgebra of $U$. We note that $U^\imath$ is not a Hopf subalgebra, but rather a right coideal subalgebra of $U$. The pair $(U,U^\imath)$ forms a quantum symmetric pair, in the sense of \cite{Le03}, corresponding to a quasi-split symmetric pair of Satake type AIII. The subalgebra $U^\imath$ is referred to as an {\em $\imath$-quantum group}.

We observe that the involutions $\psi$ and $\psi^\imath$ are not compatible with the embedding $\imath$. However, there is a remarkable element $\Upsilon$ \cite[Chapter 2.3]{BaoWang} that lies in a certain topological completion of $U$ and that intertwines these involutions in the following sense:
\begin{align}\label{id:interw}
\imath(\psi^\imath(u))\Upsilon=\Upsilon\psi(\imath(u)),\quad \imath(u)\Upsilon=\Upsilon\psi(\imath(\psi^\imath(u))),\quad u\in U^\imath.
\end{align}
This element $\Upsilon$, referred to as the {\em intertwiner}, plays a crucial role in the Bao-Wang $\imath$-canonical basis theory of quantum symmetric pairs.

Following \cite{BaoWang} a $U$-module equipped with an anti-linear bar involution compatible with the bar involution $\psi$ will be called a {\em $U$-involutive module}. We shall denote such an involution on the module also by $\psi$. Similarly, a $U^\imath$-involutive $U^\imath$-module is defined.

Given a $U$-involutive module with involution $\psi$, we can define
\begin{align}\label{form:iota:invo}
\psi^\imath(m):=\Upsilon\left(\psi(m)\right),\quad m\in M.
\end{align}
It follows from \eqref{id:interw} that $M$ is $U^\imath$-involutive with involution $\psi^\imath$. We note that the element $\Upsilon$ is invertible in a completion of $U$.

\subsection{Fock space $\bbT^{m|n}$}\label{sec:tmn}

Let $\bbV$ be the natural $U$-module with linear basis $v_r$, $r\in\hf+\Z$. The action of $U$ is given by
\begin{align*}
E_{\alpha_i}v_r=\delta_{i+\hf,r}v_{r-1},\quad F_{\alpha_i}v_r=\delta_{i-\hf,r}v_{r+1},\quad K_{\alpha_i}v_r=q^{\delta_{i-\hf,r}-\delta_{i+\hf,r}}v_r.
\end{align*}
Let $\bbW$ be the restricted dual with normalized dual basis $\{w_r|r\in\hf+\Z\}$ and $U$-action given by
\begin{align*}
E_{\alpha_i}w_r=\delta_{i-\hf,r}w_{r+1},\quad F_{\alpha_i}w_r=\delta_{i+\hf,r}w_{r-1},\quad K_{\alpha_i}w_r=q^{-\delta_{i-\hf,r}+\delta_{i+\hf,r}}w_r.
\end{align*}

A fundamental property of Lusztig's canonical basis theory for (finite rank) quantum groups is that tensor products of involutive modules can be made into involutive modules \cite[Chapter 27]{Lu} by means of the quasi-$\mc R$-matrix $\Theta$ \cite[Chapter 4]{Lu}. (For the infinite-rank quantum group $U$ such a $\Theta$ is constructed in \cite[Section 3.1]{CLW2} based on \cite{Lu}.) To be precise, for $U$-involutive $U$-modules $M$ and $N$ with involutions $\psi_M$ and $\psi_N$, respectively, the $U$-module $M\otimes N$ is $U$-involutive with involution:
$$\psi(m\otimes n):=\Theta(\psi_M(m)\otimes\psi_N(n)).$$

The $U$-modules $\bbV$ and $\bbW$ are both $U$-involutive with anti-linear involutions $\psi_{\bbV}$ and $\psi_\bbW$ determined by $\psi_{\bbV}(v_r)=v_r$ and $\psi_\bbW(w_r)=w_r$, $r\in\hf+\Z$, respectively.
It follows therefore that the $U$-module $\bbT^{m|n}:=\bbV^{\otimes m}\otimes\bbW^{\otimes n}$, or rather a certain topological completion $\wbT^{m|n}$, is $U$-involutive. Using the formula \eqref{form:iota:invo}, we can therefore make $\wbT^{m|n}$ into a $U^\imath$-involutive $U^\imath$-module.
We shall denote the involution on $\wbT^{m|n}$ also by $\psi^\imath$ so that we have $\psi^\imath(uv)=\psi^\imath(u)\psi^\imath(v)$, for $u\in U^\imath$ and $v\in\wbT^{m|n}$.

\subsection{$\imath$-Canonical bases}\label{sec:icb}

Let $\bbI^{m|n}$ denote the set of functions from $\{1,2,\cdots,m,\bar{1},\cdots,\bar{n}\}$ to $\hf+\Z$. For each $f\in\bbI^{m|n}$ we define
\begin{align*}
M_f:=v_{f(1)}\otimes \cdots\otimes v_{f(m)}\otimes w_{f(\bar{1})}\otimes \cdots\otimes w_{f(\bar{n})}.
\end{align*}
The set $\{M_f|f\in\bbI^{m|n}\}$ is a topological basis for $\wbT^{m|n}$.

We have a bijection between $\bbI_{m|n}$ and $\Lambda^{\texttt{int}}$, the set of integer weights for the Lie superalgebra $\mf{osp}(2m+1|2n)$ (see \eqref{int:wt:B}).
This bijection allows us to transfer notions that make sense for weights  such as typical, dominant, anti-dominant, Bruhat order $\succeq$ etc.~to corresponding notions on $\bbI_{m|n}$. The precise bijection is not important for the subsequent discussion in this section. In this section, we shall only need the fact that the Bruhat order on weights induces a partial order $\succeq$ on $\bbI_{m|n}$.

As we have seen earlier, $\bbT^{m|n}$ is a $U^\imath$-involutive $U^\imath$-module. Recall that we have a right action of the Hecke algebra $\mc H_{B_m}$ (corresponding to the Weyl group $W_{B_m}$ of the Lie algebra of type $B_m$) on $\bbV^{\otimes m}$ (see, e.g., \cite[(5.2)]{BaoWang} for the explicit action). According to \cite[Theorem 5.8]{BaoWang} this action commutes with the action of $U^\imath$ on $\bbV^{\otimes m}$. There is also a right action of the Hecke algebra $\mc H_{A_n}$ (corresponding to the symmetric group $\mf S_n$) on $\bbW^{\otimes n}$ (see, e.g., \cite[(2.9)]{CLW2}) that commutes with the action of $U$ by \cite[Proposition 3]{Jim}. We recall that Hecke algebras have natural anti-linear bar involutions determined by $\ov{H}_i:=H_i^{-1}$, where the $H_i$ are the usual generators satisfying $(H_i-q^{-1})(H_i+q)=0$ and the corresponding braid relations. Furthermore, in \cite[Theorem 5.8]{BaoWang} and \cite[Proposition 3]{Jim} it is shown that the involutive module structures on $\bbV^{\otimes m}$ and $\bbW^{\otimes n}$ are compatible with the bar involutions on the respective Hecke algebras. We have the following.

\begin{prop}\label{prop:comp:inv}
The action of $U^\imath$ and that of $\mc H_{B_m}\times \mc H_{A_n}$ on $\bbT^{m|n}$ commute with each other. Furthermore, these actions are compatible with their respective involutions, i.e., we have for $u\in U^\imath$, $t\in\bbT^{m|n}$, and $H\in\mc H_{B_m}\times \mc H_{A_n}$:
\begin{align*}
\psi^\imath(u t H)=\psi^\imath(u)\psi^\imath(t) \ov{H}.
\end{align*}
\end{prop}

\begin{proof}
{
Let $x\in \mathbb V^{\otimes m}$, $y\in\mathbb W^{\otimes n}$, $u\in U^\imath$, $h_1\in \mc H_{B_m}$ and $h_2\in\mc H_{A_n}$. We have $\Delta(u)=\sum_{i}u_1^i\otimes u_2^i$, with $u_1^i\in U^\imath$ and $u_2^i\in U$. We have
\begin{align*}
\left(\Delta(u)(x\otimes y)\right)(h_1\otimes h_2)&=\left(\sum_i u_1^i x\otimes u_2^i y\right)(h_1\otimes h_2)
=\sum_i \left(u_1^i x\right)h_1\otimes \left(u_2^i y\right)h_2\\
&= \sum_iu_1^i \left(xh_1\right)\otimes u_2^i \left(yh_2\right) = \Delta(u)\left((x\otimes y)(h_1\otimes h_2)\right),
\end{align*}
where the penultimate identity above follows from \cite[Theorem 5.8]{BaoWang} and \cite[Proposition 3]{Jim}. This implies that the actions of $U^\imath$ and $\mc H_{B_m}\otimes \mc H_{A_n}$ on $\bbT^{m|n}$ commute.

By \cite[Proposition 3.7 and Remark 3.14]{BaoWang} the involution $\psi^\imath$ on $\bbT^{m|n}$ can be written in the following form as an involution on the tensor product of the $U^\imath$-involutive module $\mathbb V^{\otimes m}$ and the $U$-involutive module $\mathbb W^{\otimes n}$ using the quasi-$\mc K$-matrix $\Theta^\imath$:
\begin{align}\label{ibar:tensor}
\psi^\imath(x\otimes y)=\Theta^\imath(\psi^\imath(x)\otimes\psi(y)),
\end{align}
where $\Theta^\imath:=\Delta(\Upsilon)\Theta(\Upsilon^{-1}\otimes 1)$. It is shown in \cite[Chapter 3.3]{BaoWang} that $\Theta^\imath$ is an element in (the completion of) $U^\imath\otimes U$. We shall write $\Theta^\imath=\sum_i \theta_1^i\otimes\theta_2^i$, with $\theta_1^i\in U^\imath$ and $\theta_2^i\in U$.

Now, for $u\in U^\imath$, we have clearly $\psi^\imath(u\cdot x\otimes y)=\psi^\imath(u)\cdot \psi^\imath(x\otimes y)$, since $\bbT^{m|n}$ is a $U^\imath$-involutive module by \cite[Proposition 3.13]{BaoWang}.

On the other hand, for $x$, $y$, $h_1$ and $h_2$ as above, we have, again using \cite[Theorem 5.8]{BaoWang} and \cite[Proposition 3]{Jim}:
\begin{align*}
\psi^\imath\left(\left(x\otimes y\right)\left(h_1\otimes h_2\right)\right)&=\Theta^\imath\left(\psi^\imath(xh_1)\otimes\psi(yh_2)\right) = \Theta^\imath\left(\psi^\imath(x)\ov{h}_1\otimes\psi(y)\ov{h}_2\right)\\
=\sum_i &\theta_1^i(\psi^\imath(x)\ov{h}_1)\otimes\theta_2^i(\psi(y)\ov{h}_2) = \sum_i \left(\theta_1^i\psi^\imath(x)\right)\ov{h}_1\otimes\left(\theta_2^i\psi(y)\right)\ov{h}_2\\
=\Theta^\imath&\left(\psi^\imath(x)\otimes\psi(y)\right)(\ov{h}_1\otimes\ov{h}_2) = \psi^\imath\left(x\otimes y\right)(\ov{h}_1\otimes\ov{h}_2).
\end{align*}
This concludes the proof.
}
\end{proof}

One has the following important property of the involution $\psi^\imath:\wbT^{m|n}\rightarrow\wbT^{m|n}$ \cite[Lemma 9.7]{BaoWang}:
\begin{align*}
\psi^\imath(v_f)\in v_f+\sum_{g\prec f}\Z[q,q^{-1}]v_g\subseteq \widehat{\bbT}^{m|n}.
\end{align*}
Since the partial order $\preceq$ satisfies the finite interval property by \cite[Lemma 8.4]{BaoWang} we can apply \cite[Lemma 24.2.1]{Lu} and conclude the existence of $\psi^\imath$-invariant bases:
\begin{prop}
There exist unique $\psi^\imath$-invariant bases $\{T_f|f\in\bbI^{m|n}\}$ and $\{L_f|f\in\bbI^{m|n}\}$ of the form
\begin{align*}
&T_f=M_f+\sum_{g\prec f}t_{gf}(q)M_g,\quad t_{gf}(q)\in q\Z[q],\\
&L_f=M_f+\sum_{g\prec f}\ell_{gf}(q)M_g,\quad\ell_{gf}(q)\in q^{-1}\Z[q^{-1}].\\
\end{align*}
\end{prop}

We set $t_{ff}=\ell_{ff}=1$ and $t_{gf}=\ell_{gf}=0$, for $g\not\preceq f$.

\subsection{$q$-symmetric tensors}\label{sec:qsymm}

For a finite Weyl group $W$ of a simple Lie algebra, we recall the well-known Chevalley-Solomon formula \cite[Corollary 2.3]{Sol66}:
\begin{align*}
\sum_{\tau\in W}q^{\ell(\tau)}=\prod_{i=1}^r\left(1+q+\cdots+q^{e_i}\right),
\end{align*}
where $r$ is the rank of $W$ and $e_1,\cdots,e_r$ are the exponents of the corresponding Lie algebra. This implies the following formula:
\begin{align*}
\sum_{\tau\in W}q^{\ell(w_0)-\ell(\tau)}=\prod_{i=1}^r[e_i+1],
\end{align*}
where $[s]$ is the quantum integer $\frac{q^s-q^{-s}}{q-q^{-1}}$, for $s\in\bbN$.  We shall write $[W]:=\prod_{i=1}^r[e_i+1]$.
Now, suppose that $W$ is a Weyl group of a reductive Lie algebra so that $W=W^1\times \cdots \times W^t$, where each $W^j$, $1\le j\le t$, is a Weyl group of a simple Lie algebra. We shall use the following notation:
\begin{align*}
[W]:=\prod_{j=1}^t[W^j].
\end{align*}

Let $W_\zeta$ be a (parabolic) subgroup of $W_{B_m}\times\mf{S}_n$ generated by simple reflections. Indeed, $W_\zeta$ is a product of Weyl groups of the form  $W_\zeta= W^1\times \cdots \times W^s\times W^{s+1}\times\cdots \times W^{s+t}$, where $W^1$ is either of type $A$ or $B$, and the remaining $W^i$s are all of type $A$. Furthermore, the sum of the ranks of the $W^i$s, $i=1,\cdots,s$ (respectively, $i=s+1,\cdots,s+t$) equals $m$ (respectively, equals $n$).
Let $\mc H_\zeta$ be the Hecke subalgebra of the Hecke algebra $\mc{H}_{B_m}\times \mc H_{A_n}$ associated with $W_\zeta$. Let
\begin{align}\label{eq:q:symm}
S_\zeta:=\sum_{\sigma\in W_\zeta} q^{\ell(w_0^\zeta)-\ell(\sigma)}H_\sigma = \prod_{i=1}^{s+t}S_i,
\end{align}
where $w_0^\zeta$ is the longest element in $W_\zeta$, $S_i=\sum_{\tau\in W^i} q^{\ell(w_0^i)-\ell(\tau)}H_\tau$, and $w_0^i$ is the longest element in $W^i$. The bar involution on $\mc{H}_{B_m}\times \mc H_{A_n}$ restricts to the bar involution on $\mc{H}_\zeta$, and we have that
\begin{align*}
\ov{S_i}=S_i\quad \text{and}\quad S_i H_\tau=q^{-\ell(\tau)}S_i=H_\tau S_i ,\ \tau\in W^i,\\
\ov{S_\zeta}=S_\zeta\quad \text{and}\quad S_\zeta H_\sigma=q^{-\ell(\sigma)}S_\zeta=H_\sigma S_\zeta ,\ \sigma\in W_\zeta.
\end{align*}

Consider the following $\bbQ(q)$-linear subspace of $\bbT^{m|n}$:
\begin{align}\label{eq:Fock:z}
\mathbb T^{m|n}_\zeta:=\mathbb T^{m|n} S_\zeta.
\end{align}
Then $\mathbb T^{m|n} S_\zeta$ is a $U^\imath$-submodule by Proposition \ref{prop:comp:inv}.
Let $\bbI^{m|n}_{\zeta-}$ denote the set of elements in $\bbI^{m|n}$ that are anti-dominant with respect to $W_\zeta$. For $f\in\Z^{m|n}_{\zeta-}$, we let $W_f:=\{\sigma\in W_\zeta|f\cdot\sigma=f\}$ be the stabilizer subgroup inside $W_\zeta$ and $W^f$ the set of shortest length representatives of the coset $W_f\backslash W_\zeta$.

Define the following sets of monomial bases for $\mathbb T^{m|n}_\zeta$:
\begin{align}\label{eq:Nf}
\widetilde{N}_f:=M_f S_\zeta;\qquad \widetilde{M}_f:=\frac{1}{[W_f]}M_f S_\zeta,\quad f\in\bbI^{m|n}_{\zeta-}.
\end{align}

From the formulas of the action of the Hecke algebra on $M_f$ it follows that the basis elements $\widetilde{N}_f$ lie in the $\Z[q,q^{-1}]$-span of $\{M_g|g\in\bbI^{m|n}\}$. The $\widetilde{M}_f$s also lie in the $\Z[q,q^{-1}]$-span of $\{M_g|g\in\bbI^{m|n}\}$, which can be seen as follows. Suppose that the Hecke algebra of $W^i$ acts on $\mathbb V^{\otimes m_i}$, for $i=1,\cdots,s$, and the Hecke algebra of $W^j$ acts on $\mathbb W^{\otimes n_{j}}$, $j=s+1,\cdots,s+t$. We have $\sum_{i}m_i=m$ and $\sum_{j}n_j=n$. The element
\begin{align*}
\widetilde{M}_{f_i}=\frac{1}{[W_f\cap W^i]}M_{f_i}S_i,
\end{align*}
where $f_i$ is $f$ restricted to the interval corresponding to $W^i$, is an $\imath$-canonical basis element for $i=1$ and a canonical basis element for $i>1$. Thus, by \cite[Theorem 4.20]{BaoWang} in the case of $i=1$ and \cite[Theorem 27.3.2]{Lu} in the case $i>1$, it lies in the $\Z[q,q^{-1}]$-span of corresponding $M_{g_i}$, where the notation of $g_i$ is self-explanatory. Now, we have
\begin{align*}
\widetilde{M}_f=\widetilde{M}_{f_1}\otimes\cdots\otimes\widetilde{M}_{f_{s+t}},
\end{align*}
and hence it lies in the $\Z[q,q^{-1}]$-span of $M_{g}$, $g\in\bbI^{m|n}$.

For $f\in\Z^{m|n}_{\zeta-}$, we also define
\begin{align}\label{def:Nf}
N_f:=\frac{[W_\zeta]}{[W_f]}\widetilde{N}_f.
\end{align}

Suppose that $f\in\bbI^{m|n}$, but not necessarily in $\bbI^{m|n}_{\zeta-}$. Let $\tau\in W_\zeta$ be of shortest length such that $f\cdot\tau\in\bbI^{m|n}_{\zeta-}$. Then, we have
\begin{align}\label{eq:MfS}
M_f S_\zeta= M_{f\cdot\tau} H_{\tau^{-1}} S_\zeta = q^{-\ell(\tau)} M_{f\cdot\tau}S_\zeta = q^{-\ell(\tau)} \widetilde{N}_{f\cdot\tau}.
\end{align}
Similarly, we have the identity $\frac{1}{[W_f]}M_f S_\zeta=q^{-\ell(\tau)}\widetilde{M}_{f\cdot\tau}$.

\begin{lem} The bar involution on $\widehat{\mathbb T}^{m|n}$ restricts to a bar involution on $\widehat{\mathbb T}^{m|n}_\zeta$ such that,
for $f\in\Z^{m|n}_{\zeta-}$, we have:
\begin{itemize}
\item[(1)] $\psi^\imath\left({\widetilde{N}}_f\right)\in \widetilde{N}_f+\sum_{g\prec f}\Z[q,q^{-1}] \widetilde{N}_g$,
\item[(2)] $\psi^\imath\left({\widetilde{M}}_f\right)\in \widetilde{M}_f+\sum_{g\prec f}\Z[q,q^{-1}] \widetilde{M}_g$,
\item[(3)] $\psi^\imath\left({N}_f\right)\in {N}_f+\sum_{g\prec f}\Z[q,q^{-1}] {N}_g.$
\end{itemize}
\end{lem}

\begin{proof}
For $f\in\bbI^{m|n}_{\zeta-}$, suppose that
$\psi^\imath(M_f) = M_f+\sum_{g\prec f}r_{gf}M_g$ with $r_{gf}\in\Z[q,q^{-1}]$.
Then, by Proposition \ref{prop:comp:inv} and \eqref{eq:MfS}, we have
\begin{align*}
\psi^\imath({\widetilde{N}}_f)&=\psi^\imath({M_f S_\zeta})=\psi^\imath({M}_f) S_\zeta = M_fS_\zeta+\sum_{g\prec f}r_{gf}M_gS_\zeta
=\widetilde{N}_f + \sum_{g\prec f; g\in\bbI^{m|n}_{\zeta -}}r'_{gf} \widetilde{N}_g.
\end{align*}
Since $r_{gf}\in\Z[q,q^{-1}]$, we have $r'_{fg}\in\Z[q,q^{-1}]$ by \eqref{eq:MfS}. This proves Part (1).

We know that each $\widetilde{M}_f$ is a decomposable tensor, where the first factor is either an $\imath$-canonical or canonical basis element and the remaining ones are canonical basis elements in the respective tensor powers of either $\mathbb V$ or $\mathbb W$. Now, the $\Z[q,q^{-1}]$-span of such first tensor factors of that form is invariant under the action of integral form of $U^\imath$, while the $\Z[q,q^{-1}]$-span of the remaining tensor factors are invariant under the action of the integral form of $U$. From the formula \eqref{ibar:tensor}, and the integrality of both $\Upsilon$ and $\Theta$ it follows that $\psi^\imath\left({\widetilde{M}}_f\right)$ lies in the $\Z[q,q^{-1}]$-space of the ${\widetilde{M}}_g$s. This gives Part (2).

Part (3) now follows from Part (2).
\end{proof}

Now we can apply
 \cite[Lemma 24.2.1]{Lu} and obtain the following.

\begin{prop}
The $\bbQ(q)$-vector space $\widehat{\mathbb T}^{m|n}_\zeta$ has unique
$\psi^\imath$-invariant topological bases of the form
\begin{align*}
\{\mc{ T}'_f|f\in\bbI^{m|n}_{\zeta-}\},\quad\{\mc{ T}_f|f\in\bbI^{m|n}_{\zeta-}\}\quad\text{and}\quad\{\mc{ L}_f|f\in\bbI^{m|n}_{\zeta-}\}
\end{align*}
such that
\begin{align*}
\mc T'_f=\widetilde{M}_f+\sum_{g\prec f}\texttt{t}'_{gf}(q)\widetilde{M}_g,
\quad
\mc T_f={N}_f+\sum_{g\prec f}\texttt{t}_{gf}(q){N}_g,
 \quad
\mc L_f=\widetilde{N}_f+\sum_{g\prec f}\texttt{l}_{gf}(q)\widetilde{N}_g,
\end{align*}
with $\texttt{t}'_{gf}(q),\texttt{t}_{gf}{(q)}\in q\Z[q]$, and $\texttt{l}_{gf}(q)\in
q^{-1}\Z[q^{-1}]$, for $g\prec f$ in $\Z^{m|n}_{\zeta-}$.
\end{prop}

As usual, we adopt the convention
$\texttt{t}'_{ff}(q)=\texttt{t}_{ff}(q)=\texttt{l}_{ff}(q)=1$, $\texttt{t}'_{gf}=\texttt{t}_{gf}=\texttt{l}_{gf}=0$ for $g\npreceq f$.

\subsection{$\imath$-Canonical Bases on $q$-symmetric tensors}\label{sec:can:last}

This section compares the bases $\{T_f\}$ and $\{L_f\}$ with the bases $\{\mathcal T_f\}$ and $\{\mathcal L_f\}$. Since this part is analogous to \cite[Section 3.7]{CCM} and follows along the same line of arguments, we shall be brief and only summarize the main results.

\begin{prop}\label{pi:can:basis}
Let $f,g\in\bbI^{m|n}_{\zeta-}$ and $w_0^\zeta$ be the longest element in $W_\zeta$. We have
\begin{align*}
\mc T'_f = T_{fw_0^\zeta},\qquad \mc T'_fS_\zeta=\mc T_f,\qquad
L_fS_\zeta = \mc{L}_f.
\end{align*}
In particular, we have $\texttt{t}_{gf}=\texttt{t}'_{gf} =t_{g\cdot w_0^\zeta,f\cdot w_0^\zeta}$ and $\texttt{l}_{gf} = \sum_{x\in W^g}\ell_{g\cdot x,f}q^{-\ell(x)}$, where $W^g$ is the set of shortest length representatives $W_f\cap W_\zeta\backslash W_\zeta$.
\end{prop}

Let $\phi_\zeta:\widehat{\mathbb T}^{m|n}\rightarrow\widehat{\mathbb T}^{m|n}_\zeta$ be the $U^\imath$-module homomorphism defined by
\begin{align}\label{def:phiz}
\phi_\zeta(M_f):=M_fS_\zeta,\qquad f\in\bbI^{m|n}.
\end{align}
We have the following.

\begin{thm}\label{thm:szeta:fock} Let $\phi_\zeta:\widehat{\mathbb T}^{m|n}\rightarrow\widehat{\mathbb T}^{m|n}_\zeta$ be the map in \eqref{def:phiz}. For $f\in\bbI^{m|n}$, let $\tau$ be the minimum length element in $W_\zeta$ such that $f\cdot\tau\in\bbI^{m|n}_{\zeta-}$. Then we have
\begin{itemize}
\item[(1)] $\phi_\zeta\left( M_f\right)=q^{-\ell(\tau)}\widetilde{N}_{f\cdot\tau}$ and $\phi_\zeta(\widetilde{M}_f)=q^{-\ell(\tau)}N_f$.
\item[(2)] $\phi_\zeta(T_{f\cdot w_0^\zeta})=\mc T_f$, for $f\in\bbI^{m|n}_{\zeta-}$.
\item[(3)] $\phi_\zeta\left(L_f\right)=\begin{cases}
\mc L_f,&\text{ if } f\in\bbI^{m|n}_{\zeta-},\\
0,&\text{ otherwise.}
\end{cases}$
\end{itemize}
\end{thm}

\subsection{The quantum symmetric pair $(U,U^\jmath)$ and $\jmath$-canonical bases}

A different coideal subalgebra $U^\jmath\subseteq U$ is introduced in \cite[Chapter 6.1]{BaoWang}. The pair $(U,U^\jmath)$ forms a (different) quasi-split quantum symmetric pair of type AIII. A parallel theory of canonical and dual canonical bases, called $\jmath$-canonical basis theory in loc.~cit., based on a corresponding anti-linear involution on $U^\jmath$ is then developed. In particular, it follows that on a similar Fock space, which we denote by $^\jmath\wbT^{m|n}$, one has $\jmath$-canonical and dual $\jmath$-canonical basis. Here, a difference is that in this new Fock space the standard bases in both $\bbV$ and $\bbW$ are now indexed by integers rather than half-integers so that set of the standard monomial basis is now in one-to-one correspondence with $^\jmath\bbI_{m|n}$, the set of integer-valued functions of $\{1,\cdots,m;\bar{1},\cdots,\bar{n}\}$. In \cite[Chapter 6.7]{BaoWang} a version of $(U^\jmath,\mc H_{B_m})$-duality on $\bbV^{\otimes m}$ was established, which then allows us to obtain a corresponding analogue of Proposition \ref{prop:comp:inv} in this setting.  This in turn enables us to construct the corresponding $q$-symmetrized Fock space, denoted by $^\jmath\bbT^{m|n}_\zeta$, for $\zeta\in \ch\frakn_\oa^+$, and its completion $^\jmath\wbT^{m|n}_\zeta$, which gives rise to a canonical $U^\jmath$-epimorphism ${^\jmath}\phi_\zeta:{^\jmath}\widehat{\mathbb T}^{m|n}\rightarrow{^\jmath}\widehat{\mathbb T}^{m|n}_\zeta$. A counterpart of the intertwiner for $(U,U^\jmath)$ is given \cite[Chapter 6.2]{BaoWang} with its integral property established in \cite[Chapter 6.5]{BaoWang}. This now provides us with all the necessary ingredients to repeat the constructions in Sections \ref{sec:QSP:def}--\ref{sec:can:last} for the quantum symmetric pair $(U,U^\jmath)$, leading to the construction of $\jmath$-canonical and dual $\jmath$-canonical bases on $q$-symmetrized Fock space $^\jmath\wbT^{m|n}_\zeta$ and then finally to a counterpart of Theorem \ref{thm:szeta:fock} in this setting.

\section{Category $\mc O$ for Ortho-symplectic Lie Superalgebras}\label{sec:ortho}

In this section we let $\g$ be an ortho-symplectic Lie superalgebra. We shall recall the solution of the irreducible character problem for the ortho-symplectic Lie superalgebras in $\mc O_\Z$ of \cite{BaoWang, Bao}. This will be used in the subsequent section.

\subsection{Lie superalgebra of type $B$}
For $\g=\mf{osp}(2m+1|2n)$ we shall consider the following set of simple roots $\Pi=\{-\epsilon_1,\epsilon_1-\epsilon_2,\cdots,\epsilon_{m-1}-\epsilon_m,\epsilon_m-\delta_1,\delta_1-\delta_2,\cdots, \delta_{n-1}-\delta_n\}$ so that we have the following Dynkin diagram:
\vskip0.5cm
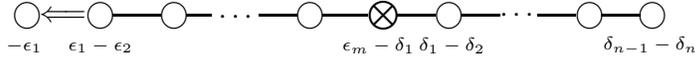
\begin{center}
\hskip -3cm \setlength{\unitlength}{0.16in}
\begin{picture}(27,2)
\put(8,2){\makebox(0,0)[c]{$\bigcirc$}}
\put(5.6,2){\makebox(0,0)[c]{$\bigcirc$}}
\put(10.4,2){\makebox(0,0)[c]{$\bigcirc$}}
\put(14.85,2){\makebox(0,0)[c]{$\bigcirc$}}
\put(17.25,2){\makebox(0,0)[c]{$\bigotimes$}}
\put(19.4,2){\makebox(0,0)[c]{$\bigcirc$}}
\put(24,2){\makebox(0,0)[c]{$\bigcirc$}}
\put(26.05,2){\makebox(0,0)[c]{$\bigcirc$}}
\put(21.7,2){\makebox(0,0)[c]{$\cdots$}}
\put(8.4,2){\line(1,0){1.55}} \put(10.82,2){\line(1,0){0.8}}
\put(13.2,2){\line(1,0){1.2}} \put(15.28,2){\line(1,0){1.45}}
\put(17.7,2){\line(1,0){1.25}}
\put(19.8,2){\line(1,0){1.25}}
\put(22.3,2){\line(1,0){1.25}}
\put(24.4,2){\line(1,0){1.25}}
\put(6,1.8){$\Longleftarrow$}
\put(12.5,1.95){\makebox(0,0)[c]{$\cdots$}}
\put(5.5,1){\makebox(0,0)[c]{\tiny$-\epsilon_{1}$}}
\put(8,1){\makebox(0,0)[c]{\tiny$\epsilon_1-\epsilon_2$}}
\put(17.1,1){\makebox(0,0)[c]{\tiny$\epsilon_m-\delta_1$}}
\put(19.5,1){\makebox(0,0)[c]{\tiny$\delta_1-\delta_2$}}
\put(26,1){\makebox(0,0)[c]{\tiny$\delta_{n-1}-\delta_n$}}
\end{picture}
\end{center}
Let $\frakb$ denote the corresponding Borel subalgebra with the set of positive roots $\Phi^+$ consisting of \begin{align*}
\{\pm\epsilon_i-\epsilon_j,-\epsilon_k|i<j\}\cup\{\pm\delta_k-\delta_l|k< l\}\cup\{-2\delta_k\}\cup\{-\delta_k\}\cup\{\pm\ep_i-\delta_k\}.
\end{align*}
The Weyl vector equals
\begin{align*}
\rho=\sum_{i=1}^m(\hf-i)\ep_i+\sum_{j=1}^n(m-j+\hf)\delta_j.
\end{align*}

For $\la\in\h^*$, we define a function $f_\la:\{1,\cdots,m;\bar{1},\cdots,\bar{n}\}\rightarrow\C$ by
\begin{align*}
f_\la(i):=\begin{cases}
(\la+\rho,\ep_i),\quad\text{ for }i= 1,\cdots,m,\\
(\la+\rho,\delta_i),\quad\text{ for }i=\bar{1},\cdots,\bar{n}.
\end{cases}
\end{align*}

Recall that $\Lambda$ denotes the set of integral weights. Denote by $\succeq$ the Bruhat order on $\Lambda$ with respect to the simple system above. We have the following subsets in $\Lambda$:
\begin{align}\label{int:wt:B}
\Lambda^{\texttt{int}}:=\sum_{i=1}^m\Z\ep_i+\sum_{j=1}^n\Z\delta_j,\qquad
\Lambda^{\texttt{hf}}:= \sum_{i=1}^m(\hf+\Z)\ep_i+\sum_{j=1}^n(\hf+\Z)\delta_j.
\end{align}
Note that if $\la\in\Lambda$, but $\la\not\in\Lambda^{\texttt{int}}\cup\Lambda^{\texttt{hf}}$, then $\la$ is a typical weight. The assignment $f\rightarrow f_\la$ restricts to bijections $\Lambda^{\texttt{int}}\rightarrow\bbI_{m|n}$ and $\Lambda^{\texttt{hf}}\rightarrow{^\jmath\bbI_{m|n}}$ .

Recall that we let $\mc O$ denote the BGG category and $\mc O_\Z$ be the integral BGG subcategory for $\g$, respectively.  We further let $\mc O^{\texttt{int}}_\Z$ and $\mc O^{\texttt{hf}}_\Z$ denote the full subcategories of $\mc O_\Z$ consisting of objects with weights lying in $\Lambda^{\texttt{int}}$ and $\Lambda^{\texttt{hf}}$, respectively.
For $\la\in\h^*$, $M(\la)$ denotes the Verma module of $\frakb$-highest weight $\la$ and $L(\la)$ denotes its unique irreducible quotient. We also recall that $T^{\mc O}(\la)$ denotes the tilting module of highest $\la$ in $\mc O$.

Let $K(\mc O)$ denote the Grothendieck group of $\mc O$ and $K(\mc O)_\bbQ=K(\mc O)\otimes_\Z\bbQ$. The subgroup $K(\mc O_\Z)_\bbQ$ has basis $\{[M(\la)]\vert \la\in\Lambda\}$. We have the self-explanatory notations of $K(\mc O^{\texttt{int}}_\Z)_\bbQ$ and $K(\mc O^{\texttt{hf}}_\Z)_\bbQ$ for the corresponding Grothendieck groups.

We have a linear isomorphism $\psi:\bbT^{m|n}_{q=1}\rightarrow {K(\mc O^{\texttt{int}}_\Z)_\bbQ}$ at $q=1$ given by $ M_{f_\la}\rightarrow[M(\la)]$, by means of which we can define a topological completion $\widehat{K}(\mc O^{\texttt{int}}_\Z)_\bbQ$ of ${K(\mc O^{\texttt{int}}_\Z)}_\bbQ$. This gives a linear isomorphism $\psi: \wbT^{m|n}_{q=1}\rightarrow\widehat{K}(\mc O^{\texttt{int}}_\Z)_\bbQ$ at $q=1$. Indeed, $\psi$ is a $U^\imath_{q=1}$-homomorphism, where the action of the generators of $U^\imath_{q=1}$ on the Fock space correspond to that of certain translation functors (constructed by tensoring with symmetric powers of the natural module) acting on $\mc O^{\texttt{int}}_\Z$ \cite[(11.3)--(11.5)]{BaoWang}. We have the following solution of the irreducible character problem in $\mc O^{\texttt{int}}_\Z$.

\begin{thm}\label{thm:BW}\cite[Theorem 11.13]{BaoWang}
The map $\psi: \wbT^{m|n}_{q=1}\rightarrow\widehat{K}(\mc O^{\texttt{int}}_\Z)_\bbQ$ defined by sending $M_{f_\la}\rightarrow[M(\la)]$ satisfies
$$\psi\left(L_{f_\la}\right)=[L(\la)],\quad\text{and}\quad \psi\left(T_{f_\la}\right)=[T^{\mc O}(\la)].$$
In particular, we have $\ch L(\la)=\sum_{\mu\preceq \la}\ell_{f_\la,f_\mu}(1)\ch M(\mu)$.
\end{thm}

\begin{rem}
For the category $\mc O^{\texttt{hf}}_\Z$ of half-integer weights of $\g$ we have a counterpart to Theorem \ref{thm:BW} above, now with $\wbT^{m|n}_{q=1}$ and $\widehat{K}(\mc O^{\texttt{int}}_\Z)_\bbQ$ replaced by $^\jmath\wbT^{m|n}_{q=1}$ and $\widehat{K}(\mc O^{\texttt{hf}}_\Z)_\bbQ$, respectively \cite[Theorem 12.5]{BaoWang}. The linear isomorphism is indeed a $U^\jmath_{q=1}$-homomorphism, where the action of the generators of $U^\jmath_{q=1}$ corresponds to that of the same translation functors acting now on $\mc O^{\texttt{hf}}_\Z$ \cite[Remark 12.4]{BaoWang}. The map again matches $\jmath$-canonical and dual $\jmath$-canonical basis elements of $^\jmath \wbT^{m|n}_{q=1}$ with tilting and irreducible objects in $\mc O^{\texttt{hf}}_\Z$, respectively.
\end{rem}

\subsection{Lie superalgebra of type $D$}\label{sec:type:D}
There is a Kazhdan-Lusztig theory for the Lie superalgebra $\g=\mf{osp}(2m|2n)$ based on different $\imath$-quantum groups. Indeed, similar Fock space realizations for the BGG categories of $\g$-modules of integer and half-integer weights were constructed in \cite{Bao}. In particular, an $\mf{osp}(2m|2n)$-counterpart of Theorem \ref{thm:BW} is obtained in\cite[Theorem 4.15]{Bao} for both integer and half-integer weight modules in $\mc O_\Z$.

\section{Serre Quotient Functor and Proof of Theorem \ref{thm::mainthmB}} \label{Sect::7}
The Serre quotient functor $\pi: {\mc O_\Z}\rightarrow \OI$ satisfies the universal property as follows. If there is an exact functor $F: {\mc O_\Z}\rightarrow \mc C$, where $\mc C$ is an abelian category, such that $F(X)=0$ for any $X\in \mc I_\zeta$. Then there is a unique exact functor $F': \OI\rightarrow \mc C$ such that $F=F'\circ \pi$; see \cite[Corollaries II.1.2 and III.1.3]{Gabriel} for more details.

In this section, we assume that $\g$ is a basic   Lie superalgebra, that is, $\g$ is one of the Lie superalgebras from \eqref{eq::Kaclist} except for the series $\mf p(n)$.  We consider the restriction of the Backelin functor $\Gamma_\zeta$ to $\mc O_\Z$. We will prove that $\Gamma_\zeta:{ \mc O_\Z\rightarrow \Gamma_\zeta(\mc O_\Z)}$ also satisfies this universal property. As a consequence, $\Gamma_\zeta$ provides an explicit realization of $\pi$, extending the results in \cite[Corollary 37]{CCM2}, where the cases of Lie superalgebras of type I, including $\mf p(n)$, are considered.  We shall now combine the results of the previous sections to complete the proof of Theorem 2.

\subsection{Functor $F_\zeta$ and Gorelik's equivalence} \label{sect::71}
Recall from \cite[Theorem 26]{Ch212} (see also \cite[Section 7.3]{CCM2}) that there is an exact functor from {$\mc O_\Z$} to $\mc W(\zeta)$:
\begin{align}
	&F_\zeta(-):= \mc L(M_0(\nu), -)\otimes_{U(\g_\oa)}M_0(\nu,\zeta): {\mc O_\Z} \rightarrow \mc W(\zeta),
\end{align}   where $\nu\in \h^\ast$ is a dominant and integral weight with the stabilizer subgroup $W_\zeta$. Here, for any $X\in \mc O$, the $\mc L(M_0(\nu), X)$ denotes the maximal $(\g,\g_\oa)$-submodule of $\Hom_\C(M_0(\nu),X)$ that is a direct sum of finite-dimensional $\g_\oa$-modules with respect to the adjoint action of $\g_\oa$.

The functor $F_\zeta(-)$ satisfies the universal property of the Serre quotient functor as described above and restricts to an equivalence between $\Opres$ and $\mc W(\zeta)$.  In the case when $\g$ is of type I, the Backelin functor $\Gamma_\zeta$ restricted to $\mc O_\Z$ and the functor $F_\zeta$  are isomorphic; see \cite[Corollary 37]{CCM}.

 We put $\g\mod_Z$ to be the category of all finitely generated $\g$-modules on which $Z(\g)$ acts locally finitely. For any central character $\chi$ of $\g$, we set $(\cdot)_\chi$ to be  the endofunctor on $\g\mod_{Z}$ of taking the largest summand of modules annihilated by some power of the kernel $\ker(\chi)$.

Let $\chi$ be a strongly typical central character with a perfect mate $\chi^0$ in the sense of \cite{Gor2}.  Denote by $\mc O_{\chi}$ and $\mc O_{\chi^0}$ the corresponding central blocks, namely,  $\mc O_{\chi}$ and $\mc O_{\chi^0}$ are the full subcategories of objects in $\mc O$ and $\mc O_\oa$ annihilated by some powers of $\ker(\chi)$ and $\ker(\chi^0)$, respectively. The following lemma is due to Gorelik  \cite{Gor2} (see also \cite[Lemma 3.1]{Co16}):
\begin{lem}[Gorelik]\label{lem:gorelik} Suppose that $\chi$ is a strongly typical central character with a perfect mate $\chi^0$.
	The following functors
	\begin{align}
	&\Res(-)_{\chi^0},~\Ind(-)_{\chi},
	\end{align} give rise to a mutually inverse equivalence of central blocks $\mc O_{\chi}$ and $\mc O_{\chi^0}$. Furthermore, this equivalence of categories preserves Verma modules.
\end{lem}

\begin{lem}  \label{lem::FstaWhi}
  Suppose that $\la, \la'\in \h^\ast$ such that
  \begin{align}
  &\Ind(M_0(\la'))_{\chi} = M(\la),\\
  &\Res(M(\la))_{\chi^{0}} = M_0(\la').
  \end{align}
  Then we have
	\begin{align}
	&\Res(M(\la,\zeta))_{\chi^0} \cong  M_0(\la',\zeta),\label{eq::41}\\
	&\Ind(M_0(\la',\zeta))_{\chi}\cong M(\la,\zeta),\label{eq::42}\\
	&F_\zeta(M(\la))\cong M(\la,\zeta). \label{eq::43}
	\end{align}
\end{lem}
\begin{proof}
	First, we calculate
	\begin{align*}
	&\Res(M(\la,\zeta))_{\chi^0} \\
	&= \Res(\Gamma_\zeta(M(\la)))_{\chi^0} \text{ by Proposition \ref{prop::3}}\\
	&=\Gamma_\zeta^0(\Res(M(\la))_{\chi^0})\\
	&=\Gamma_\zeta^0(M_0(\la'))\\
	&=M_0(\la',\zeta) \text{ by \cite[Proposition 6.9]{B}.}
	\end{align*}

Next we calculate
	\begin{align*}
	&\Ind(M_0(\la',\zeta))_{\chi} \\
	&= \Ind(\Gamma_\zeta^0(M_0(\la')))_{\chi}\\
	&=\Gamma_\zeta(\Ind(M_0(\la')_{\chi})\\
	&=\Gamma_\zeta(M(\la))\\
	&=M(\la,\zeta) \text{ by Proposition \ref{prop::3}.}
\end{align*}

Finally, we calculate

\begin{align*}
	&F_\zeta(M(\la)) = F_\zeta(\Ind(M_0(\la'))_{\chi})\\
	&=   \Ind(F_\zeta^0(M_0(\la')))_\chi  \\
	&=\Ind(M_0(\la',\zeta))_\chi 	\text{ by \cite[Proposition 5.15]{MS}.} \\
	&= M(\la,\zeta) \text{ by the isomorphism \eqref{eq::42}.}
\end{align*}
\end{proof}

\begin{cor} \label{cor::SQ}
	The functors $\Gamma_\zeta(-)$ and $F_\zeta(-)$ { from $\mc O_\Z$ to $\mc N(\zeta)$} are isomorphic. In particular, the Backelin functor	$\Gamma_\zeta(-): { \mc O_\Z\rightarrow \mc W(\zeta)}$ satisfies the universal property of Serre quotient.
\end{cor}
\begin{proof} 	Let $\la\in \h^\ast$ be a  strongly typical,	dominant, integral and regular weight. Let $\la'\in \h^\ast$ be such that  the central character $\chi^0:=\chi^0_{\la'}$ of $\g_\oa$ associated to $\la'$ is a perfect mate of $\chi:=\chi_\la$ and $\Res(M(\la))_{\chi^{0}} = M_0(\la')$. By Lemmas \ref{lem:gorelik} and \ref{lem::FstaWhi} it follows that  $$F_\zeta(M(\la))\cong M(\la,\zeta) \cong \Gamma_\zeta(M(\la)).$$
	Then we have  $\Gamma_\zeta\cong F_\zeta$ by \cite[Lemma 1]{CCM2} (see also \cite[Corollary 37]{CCM2}). The conclusion follows. 
	\end{proof}

\begin{rem}\label{rem::28} Recall the coapproximation functor $\mf j:\mc O_\Z\rightarrow \mc O^{\vpre}$ from Section \ref{Sect::coappr}. By Lemma \ref{lem::FstaWhi}, under the equivalences of categories $\Opres\cong\mc W(\zeta)\cong \OI$, we have a correspondence between the simple objects $\mf j(L(\la))\cong P(\la)/Tr(\text{rad}P(\la))$, $L(\la,\zeta)$ and $\pi(L(\la))$, and a correspondence between the proper standard objects $\ov \Delta(\la)$, $M(\la,\zeta)$ and $\pi(M(\la))$. Here $\text{rad}P(\la)$ denotes the radical of $P(\la)$.
\end{rem}

\subsection{Proof of Theorem \ref{thm::mainthmB}}  \label{Sect::thm:match:can}
 In this section, we consider the case when $\g$ is an ortho-symplectic Lie superalgebra and make connection between $\mc W(\zeta)$ and (dual) $\imath$-canonical bases on the $q$-symmetrized Fock space from Section \ref{sec:can:last}.

First, suppose that $\g=\mf{osp}(2m+1|2n)$. Let $\mc W(\zeta)^{\texttt{int}}$ be the full subcategory of $\mc W(\zeta)$ consisting of objects that have composition factors of the form $L(\la,\zeta)$, with $\la\in\Lambda^{\texttt{int}}$. Similarly, the full subcategory $\mc W(\zeta)^{\texttt{hf}}$ consists of objects that have composition factors of the form $L(\la,\zeta)$, with $\la\in\Lambda^{\texttt{hf}}$.  We note that $\Gamma_\zeta\left(P(\la)\right)$, $\la\in\Lnua$ integral, are the projective indecomposable modules in $\mc W(\zeta)$ and there are no nontrivial extension between $L(\mu)$ and $L(\nu)$ in $\mc O_\Z$ with $\mu\in\Lambda^{\texttt{int}}$ and $\nu\in\Lambda^{\texttt{hf}}$. It follows that there are no nontrivial extensions between modules from $\mc W(\zeta)^{\texttt{int}}$ and $\mc W(\zeta)^{\texttt{hf}}$, and indeed $\Gamma_\zeta(\mc O_\Z^\texttt{int})=\mc W(\zeta)^{\texttt{int}}$ and $\Gamma_\zeta(\mc O_\Z^\texttt{hf})=\mc W(\zeta)^{\texttt{hf}}$. Let $K(\mc W(\zeta)^{\texttt{int}})_\bbQ$ and $K(\mc W(\zeta)^{\texttt{hf}})_\bbQ$ denote the respective Grothendieck groups with rational coefficients.

Define the linear map $\psi_\zeta:\widehat{\bbT}^{m|n}_{\zeta,q=1} \rightarrow \widehat{K}(\mc W(\zeta)^{\texttt{int}})_\bbQ$ at $q=1$ by $\psi_\zeta\left(\widetilde{N}_{f_\la}\right):=[M(\la,\zeta)]$, for $\la\in \Lnua\cap\Lambda^{\texttt{int}}$. Here, $\widehat{K}(\mc W(\zeta)^{\texttt{int}})_\bbQ$ denotes the corresponding topological completion inherited from the topological completion of $\bbT^{m|n}_{\zeta,q=1}$ via this map.

\begin{thm}\label{thm:match:can} Let $\g=\mf{osp}(2m+1|2n)$.
\begin{itemize}
\item[(1)]
The functors $\Gamma_\zeta$ and $F_\zeta$ from $\mc O^{\texttt{int}}_\bbZ$ to $\mc W(\zeta)^{\texttt{int}}$ categorify the $U^\imath$-homomorphism $\phi_\zeta:\wbT^{m|n}_{q=1}\rightarrow\wbT^{m|n}_{\zeta,q=1}$ at $q=1$.
\item[(2)] The map $\psi_\zeta$ sends $\imath$-canonical and dual $\imath$-canonical basis elements $\mc L_{f_\la}$ and $\mc T_{f_\la}$ to the tilting and simple objects corresponding to weight $\la\in\Lambda^{{\texttt{int}}}{\cap\Lnua}$, respectively.
\end{itemize}
\end{thm}

\begin{proof}
Let $\gamma_\zeta:K(\mc O^{\texttt{int}}_\bbZ)_\bbQ\rightarrow K(\mc W(\zeta))_\bbQ$ be the map induced by the Backelin functor $\Gamma_\zeta$ on the respective Grothendieck groups. It is evident that the image lies in $K(\mc W(\zeta)^{\texttt{int}})_\bbQ$. By Proposition \ref{prop::3} and Theorem \ref{thm:szeta:fock} one verifies that $\gamma_\zeta\circ\psi$ and $\psi_\zeta\circ\phi_\zeta$ coincide on the monomial basis of $\bbT^{m|n}$ and hence $\gamma_\zeta\circ\psi=\psi_\zeta\circ\phi_\zeta$, which is precisely the commutativity of the diagram \eqref{comm:diag}. We note that $\phi_\zeta$ is a $U^\imath_{q=1}$-homomorphism and the Backelin functor commutes with translation functors. Furthermore, $\psi$ is compatible with the $U^\imath_{q=1}$-action on $\wbT^{m|n}_{q=1}$ and certain $\g$-translation functor action on $\mc O_\Z$ by \cite[Proposition 11.9]{BaoWang}. It follows therefore that the map $\psi_\zeta$ is also compatible with the $U^\imath_{q=1}$-action on $\wbT^{m|n}_\zeta$ and translation functor action on $\mc W(\zeta)^{\texttt{int}}$.

Part (1) is now a consequence of Proposition \ref{prop::3} and Theorem \ref{thm::4}(ii), \ref{thm:szeta:fock} and \ref{thm:BW}.

For $\la\in\Lambda(\zeta)\cap\Lambda^{\texttt{int}}$ we compute, using Theorems \ref{thm::4}(ii), \ref{thm:szeta:fock} and \ref{thm:BW}, that
\begin{align*}
[L(\la,\zeta)]=[\Gamma_{\zeta}(L(\la))]=\gamma_\zeta([L(\la)])=\gamma_\zeta(\psi(L_{f_\la}))=\psi_\zeta(\phi_\zeta(L_f)) =\psi_\zeta(\mc L_{f_\la}).
\end{align*}
This proves Part (2) for simple objects.

By Proposition \ref{prop:tilting}  each tilting module in $\mc O^{\vpre}$ is of the form  $T(\la)=T^{\mc O}(w_0^\zeta\cdot\la)\in\Opres$ for some $\la\in\Lnua$.
From this and the fact that the functors $\Gamma_\zeta$ and $F_\zeta$ are isomorphic to the same Serre quotient functor, and $F_\zeta$ restricts to an equivalence of categories from $\Opres$ to $\mc W(\zeta)$, it follows that $\Gamma_\zeta(T^{\mc O}(w_0^\zeta\cdot\la))$ is a tilting module in $\mc W(\zeta)$. By Theorem \ref{thm:szeta:fock}, for $\la\in\Lnua\cap\Lambda^{\texttt{int}}$, it follows that  $$\gamma_\zeta\left([T(\la)]\right)=\gamma_\zeta\left([T^{\mc O}(w_0^\zeta\cdot\la)]\right)=\gamma_\zeta\circ\psi(T_{ f_{\la}\cdot w_0^\zeta})=\psi_\zeta\circ\phi_\zeta (T_{ f_{\la}\cdot w_0^\zeta})=\psi_\zeta(\mc T_{f_\la}).$$
This completes the proof of Part (2).
\end{proof}

We can formulate the following half-integer version of Theorem \ref{thm:match:can} with similar proof.

\begin{thm}\label{thm:match:can1} Let $\g=\mf{osp}(2m+1|2n)$.
\begin{itemize}
\item[(1)]
The functors $\Gamma_\zeta$ and $F_\zeta$ from $\mc O^{{\texttt{hf}}}_\bbZ$ to $\mc W(\zeta)^{\texttt{hf}}$ categorify the $U^\jmath$-homomorphism ${^\jmath}\phi_\zeta:{^\jmath}\wbT^{m|n}_{q=1}\rightarrow{^\jmath}\wbT^{m|n}_{\zeta,q=1}$ at $q=1$.
\item[(2)] The induced linear map ${^\jmath}\psi_\zeta:{^\jmath}\widehat{\bbT}^{m|n}_{\zeta,q=1} \rightarrow \widehat{K}(\mc W(\zeta)^{\texttt{hf}})_\bbQ$ at $q=1$ that sends standard monomial basis elements to corresponding standard Whittaker module of weights in $\Lnua\cap\Lambda^{\texttt{hf}}$ matches $\jmath$-canonical and dual $\jmath$-canonical basis elements with the tilting and simple objects of weights in $\Lnua\cap\Lambda^{{\texttt{hf}}}$, respectively.
\end{itemize}
\end{thm}

\begin{rem}
As mentioned earlier in Section \ref{sec:type:D} we have Fock space realizations for the Kazhdan-Lusztig theory of type $D$ Lie superalgebras in the integral category $\mc O$ \cite{Bao} . This allows us to obtain counterparts of Theorems \ref{thm:match:can} and \ref{thm:match:can1} for the Lie superalgebra $\g=\mf{osp}(2m|2n)$ as well.
\end{rem}

\begin{rem}
Theorem \ref{thm:match:can} states that the natural correspondence between certain standard monomial basis elements of $\wbT^{m|n}_\zeta$ and standard Whittaker modules identifies the $\imath$-canonical basis in $\wbT^{m|n}_\zeta$ with tilting objects in $\mc W(\zeta)$. Indeed, using properties of $\psi^\imath$, we can construct a different bar involution on $\wbT^{m|n}$ and then on $\wbT^{m|n}_\zeta$ compatible with the reverse Bruhat order, i.e., applying the bar involution to a standard monomial basis element gives a $\Z[q,q^{-1}]$-linear combination of standard monomial basis elements corresponding to weights higher in the Bruhat ordering. Here, the completion is understood to be compatible with the reverse Bruhat ordering. In this setup, using Corollary \ref{cor:ringel:dual}, it can shown that the natural correspondence between the standard monomial basis elements and standard Whittaker modules will now identify the dual $\imath$-canonical basis elements in $\wbT^{m|n}_\zeta$ with the projective indecomposable modules in $\mc W(\zeta)$. Similar version exists for $\jmath$-canonical basis. We also have counterparts for $\g=\mf{osp}(2m|2n)$ as well.
\end{rem}

\subsection{Annihilator ideals}
For a given $\mf g$-module $M$, we denote by $\Ann_{\g}M$ the annihilator ideal of $M$. The following corollary establishes an analogue of \cite[Theorem 3.9]{Ko78} for Lie algebras and \cite[Theorem B]{Ch212} for Lie superalgebras of type I for basic Lie superalgebras:
\begin{cor} \label{cor::ann}
	We have
	\begin{align}
	&\Ann_\g(L(\la,\zeta)) = \Ann_{\g}(L(\la)),
	\end{align} for any $\la \in \Lnua$.
\end{cor}
\begin{proof}
 By \cite[Theorem A(2), Theorem 26(2)]{Ch212} we have
 \begin{align}
 &\Ann_\g(F_\zeta(L(\la))) = \Ann_{\g}(L(\la)),
 \end{align} for any $\la \in \Lnua$.
 	The corollary now follows from Theorem \ref{thm::4}(ii) and Corollary \ref{cor::SQ}.
\end{proof}

\vspace{2mm}




\end{document}